%% file: HickernellMCQMC2016Tutorial_rev2.tex
\documentclass[graybox,footinfo]{svmult}

\smartqed
\usepackage{mathptmx}       
\usepackage{helvet}         
\usepackage{courier}        
\usepackage{type1cm}        
\usepackage{graphicx}       

\usepackage{array,colortbl}
\usepackage{amsmath,amsfonts,amssymb,bm} 
\DeclareFontFamily{U}{mathx}{\hyphenchar\font45}
\DeclareFontShape{U}{mathx}{m}{n}{
	<5> <6> <7> <8> <9> <10>
	<10.95> <12> <14.4> <17.28> <20.74> <24.88>
	mathx10
}{}
\DeclareSymbolFont{mathx}{U}{mathx}{m}{n}
\DeclareFontSubstitution{U}{mathx}{m}{n}
\DeclareMathAccent{\widecheck}      {0}{mathx}{"71}

\input{macros}

\DeclareSymbolFont{bbold}{U}{bbold}{m}{n}
\DeclareSymbolFontAlphabet{\mathbbold}{bbold}

\usepackage{microtype} 

\usepackage[colorlinks=true,linkcolor=black,citecolor=black,urlcolor=black]{hyperref}
\usepackage{bookmark}
\makeatletter 
\providecommand*{\toclevel@author}{999}
\providecommand*{\toclevel@title}{0}
\makeatother

\usepackage{xspace}

\begin{document}
	
\newcommand{\FJHkeepvalues}{
	\edef\FJHrestorevalues{%
			\arrayrulewidth=\the\arrayrulewidth
		}%
	}
	
\FJHkeepvalues
\setlength{\arrayrulewidth}{0.5mm}
	
\newcommand{\FJHLessonZero}{The trio identity \eqref{FJH:eq:trio} decomposes the 
cubature error into a 
	product of three factors: the variation of the integrand, the discrepancy of the 
	sampling measure, and the confounding. This identity shows how the integrand and 
	the sampling measure each contribute to the cubature error.  }	

\newcommand{\FJHLessonSix}{The trio identity has four versions, \eqref{FJH:eq:dtrio}, 
\eqref{FJH:eq:rtrio}, \eqref{FJH:eq:btrio}, and \eqref{FJH:eq:rbtrio}, depending on 
whether the integrand is deterministic or Bayesian and whether the sampling measure is 
deterministic or random.  }

\newcommand{\FJHLessonTwoHalf}{The deterministic discrepancy when $\calF$ is an 
RKHS  has a simple, explicit form involving three terms. }	

\newcommand{\FJHLessonFourteen}{The formula for the Bayesian discrepancy mimics 
that for the deterministic discrepancy with $\calF$ an RKHS. }	

\newcommand{\FJHLessonSeven}{Although it has traditionally been ignored, the 
confounding helps explain why the cubature error may decay to zero 
much faster or more slowly than the discrepancy.  }

\newcommand{\FJHQOne}{How do good sampling measures, $\widehat{\nu}$, make 
the error smaller?}

\newcommand{\FJHLessonTwo}{Quasi-Monte Carlo methods replace IID data sites by 
	low discrepancy data sites, such as
	Sobol' sequences and integration lattice nodeset sequences.  The resulting sampling 
	measures have discrepancies and cubature errors that decay to zero at 
	a faster rate than in the case of IID sampling. }

\newcommand{\FJHLessonThree}{Randomizing the sampling measure may not only 
	eliminate bias, but it may help improve accuracy by avoiding the awful minority of 
	possible sampling measures. }

\newcommand{\FJHLessonFive}{The benefits of sampling measures with asymptotically 
smaller discrepancies are limited to those integrands with finite variation. }

\newcommand{\FJHQTwo}{How can the error be decreased by re-casting the problem 
with a different integrand, $f$? }

\newcommand{\FJHLessonFour}{Well-chosen variable transformations may reduce 
	cubature error by producing an integrand with a smaller variation than otherwise. }

\newcommand{\FJHLessonEight}{The cubature error for high dimensional problems can 
often be reduced by arranging for the integrand to depend primarily on those coordinates 
with lower indices. }

\newcommand{\FJHLessonThirteen}{Infinite dimensional problems may be efficiently 
solved by multi-level methods or multivariate decomposition methods, which approximate 
the integral by a sum of finite dimensional integrals. }

\newcommand{\FJHQThree}{How many samples, $n$, are required to meet a specified 
error tolerance? }

\newcommand{\FJHLessonTen}{Bayesian cubature provides data-based probabilistic 
error bounds under the assumption that the integrand is a Gaussian process. }

\newcommand{\FJHLessonEleven}{Automatic stopping criteria for 
(quasi\nobreak\mbox{-)}Monte Carlo 
simulations have been developed for integrands that lie in a cone of functions that are 
not too 
wild. }

\newcommand{\FJHCNF}{\textup{CNF}}
\newcommand{\FJHDSC}{\textup{DSC}}
\newcommand{\FJHVAR}{\textup{VAR}}

\newlength{\FJHfigheight}
\setlength{\FJHfigheight}{4 cm}
\newcommand{\FJHFigDirectory}{Hickernell} 

\newtheorem{FJHLesson}{Lesson}
\newcommand{\FJHnorm}[1]{\ensuremath{\left \lVert #1 \right \rVert}} 
\newcommand{\FJHabs}[1]{\ensuremath{\bigl \lvert #1 \bigr \rvert}}
\newcommand{\FJHbiggabs}[1]{\ensuremath{\biggl \lvert #1 \biggr \rvert}}
\newcommand{\FJHip}[3][{}]{\ensuremath{\left \langle #2, #3 \right \rangle_{#1}}}

\allowdisplaybreaks

\title*{The Trio Identity for Quasi-Monte Carlo Error}
\author{Fred J. Hickernell}
\institute{Fred J. Hickernell
\at Department of Applied Mathematics,  Illinois Institute of Technology, 10 W. 32$^{\text{nd}}$ Street, RE 208, Chicago, IL 60616, USA 
\email{hickernell@iit.edu}}
\maketitle

\abstract{Monte Carlo methods approximate integrals by sample averages 
of integrand values.  The error of Monte Carlo methods may be expressed as a trio 
identity: the product of the variation of the integrand, the discrepancy of the sampling 
measure, and the confounding.  The trio identity has different versions, depending 
on whether  the integrand is deterministic or Bayesian and whether the sampling 
measure is deterministic or random.  Although the variation and the discrepancy are 
common in the literature, the confounding is relatively unknown and under-appreciated.   
Theory and examples are used to show how the cubature error may be reduced by 
employing the low discrepancy sampling that defines quasi-Monte Carlo methods.  The 
error may also be reduced  by rewriting the integral in terms of a different integrand.  
Finally, the confounding  explains why the cubature error might  decay at a rate different 
from that of the discrepancy.} 

\section{Introduction}
Monte Carlo methods are used to approximate multivariate integrals that cannot be 
evaluated analytically, i.e.,  integrals of the form
\begin{equation}
\mu = \int_{\calX} f(\bsx) \, \nu(\D \bsx), \label{FJH:eq:INT} \tag{INT}
\end{equation}
where $f:\calX \to \R$ is a measurable function, $\calX$ is a measurable set, and $\nu$ 
is a 
\emph{probability} measure.  Here, $\mu$ is the weighted average of the integrand.  Also,
$\mu = \bbE[f(\bsX)]$, where the random variable $\bsX$ has probability measure 
$\nu$.  
Monte 
Carlo methods take the form of a weighted average of values of $f$ at a finite number of data 
sites, $\bsx_1, \ldots, \bsx_n$:
\begin{equation} \label{FJH:eq:AppxINT} \tag{MC}
\widehat{\mu} = \sum_{i=1}^n f(\bsx_i) w_i = \int_{\calX} f(\bsx) \, \widehat{\nu}(\D \bsx).
\end{equation}
The sampling measure, $\widehat{\nu}$, assigns a weight $w_i$ to the function value at  
$\bsx_i$ and lies in the vector space
\begin{equation} \label{FJH:eq:sampmeaset}
\calM_{\textup{S}} := \left \{\sum_{i=1}^n w_i \delta_{\bsx_i} : w_1, \ldots, w_n \in \R,\ 
\bsx_1, 
\ldots, \bsx_n \in \calX, \ n \in \N \right \},
\end{equation}
where  $\delta_{\bst}$ denotes a Dirac measure concentrated at point $\bst$. 
The data sites, the weights, and the sample size may be deterministic or random.  Later, 
we impose some constraints to facilitate the analysis.  

We are particularly interested in sampling measures that choose the data sites 
\emph{more cleverly} 
than independently and identically distributed (IID) with the aim of obtaining smaller 
errors for the same computational effort.  Such sampling measures are the 
hallmark of \emph{quasi-Monte Carlo methods}.  It is common to choose $w_1 = \cdots = 
w_n = 1/n$, in which case the sampling quality is determined solely by the choice of the 
data sites.

This tutorial describes how to characterize and analyze the \emph{cubature error}, $\mu 
- \widehat{\mu}$, as a trio identity:
\begin{equation} \tag{TRIO} \label{FJH:eq:trio}
\mu - \widehat{\mu}  = \FJHCNF(f,\nu - \widehat{\nu}) \, \FJHDSC(\nu - \widehat{\nu}) \, 
\FJHVAR(f),
\end{equation}
introduced by Xiao-Li Meng \cite{Men16a}.  Each term in this identity contributes to the 
error, and there are ways to decrease each. 
\begin{description}
	\item[$\FJHVAR(f)$] measures the \emph{variation} of the integrand from a typical 
	value. 
    The variation is positively homogeneous, i.e., $\FJHVAR(cf)  = 
	\FJHabs{c} \FJHVAR(f)$.  The variation is \emph{not} the variance.  Expressing $\mu$ 
	in 
	terms 
	of a different integrand by means of a variable transformation may decrease the 
	variation.
	\item [$\FJHDSC(\nu - \widehat{\nu})$] measures the \emph{discrepancy} of the 
	sampling measure 
	from the probability measure that defines the integral. The convergence rate of the 
	discrepancy to zero as $n \to \infty$ characterizes the quality of the sampling 
	measure.
	\item [$\FJHCNF(f,\nu - \widehat{\nu})$] measures the \emph{confounding} between 
	the 
	integrand and the difference between the measure defining the integral and the 
	sampling measure.  The magnitude of the confounding is bounded by one in 
	some settings and has an 
	expected square value of one in other settings.  When the convergence rate of 
	$\widehat{\mu} \to 
	\mu$ differs from the convergence rate of $ \FJHDSC(\nu - \widehat{\nu}) \to 0$, the 
	confounding is behaving unusually.
\end{description}

There are four versions of the trio identity corresponding to different models for the 
integrand and for the sampling measure as depicted  in Table 
\ref{FJH:tab:TrioMatrix}.  The integrand may be an arbitrary (deterministic) element of  a 
Banach space  
or it may be a Gaussian stochastic process.  The 
sampling measure may be an arbitrary  (deterministic) element of $\calM_{\textup{S}}$ or 
chosen 
randomly. Here we derive and explain these four  different versions of the trio 
identity and draw a baker's dozen of key lessons, which are repeated at the end of this 
article.

\begin{table}
	\setlength{\arrayrulewidth}{0.5mm}
	\centering
	\caption{Different versions of the trio identity \label{FJH:tab:TrioMatrix}}
	\begin{tabular}{r@{\quad}|c@{\quad}c@{\qquad}c}
		\multicolumn{2}{c}{}& \multicolumn{2}{c}{Sampling Measure, $\widehat{\nu}$}\\
		Integrand, $f$ && Deterministic & Random  
		\\    \cline{2-4} \\ [-1ex]
		Deterministic && Deterministic $={}^{\textup{D}}$ & Randomized $={}^{\textup{R}}$ 
		\\ 
		[1ex]
		Gaussian Process && Bayesian $={}^{\textup{B}}$ & Bayesian Randomized 
		$={}^{\textup{B}\textup{R}}$	
\end{tabular}
\end{table}

\begin{FJHLesson} \FJHLessonZero \end{FJHLesson}

\section{A Deterministic Trio Identity for Cubature Error} \label{FJH:sec:dettrio}

We start by generalizing the error bounds of 
Koksma \cite{Kok42} and Hlawka \cite{Hla61}. See also the monograph of Niederreiter 
\cite{Nie92}.  Suppose that the integrand lies in some Banach space, 
$(\calF,\FJHnorm{\cdot}_{\calF})$, 
where function evaluation at any point  in the 
domain,  $\calX$, is a 
bounded, linear functional.  This means that $\sup_{f \in \calF} 
\FJHabs{f(\bst)}/\FJHnorm{f}_{\calF} < 
\infty$ for all $\bst \in \calX$ and that $\int_{\calX} f(\bsx) \, \delta_{\bst}(\D\bsx) 
=f(\bst)$ 
for all $f \in \calF, \ \bst \in \calX$.  For example, one might choose $\calF = C[0,1]^d$, 
but $\calF = L^2[0,1]^d$ is unacceptable. 
Let $T :\calF \to \R$ be some bounded 
linear functional providing a typical value of $f$, e.g.,   $T(f) 
= f(\bsone)$ or $T(f) = \int_{\calX} f(\bsx) \, \nu(\D\bsx)$.  If $\{T(f) : f \in 
\calF\} \ne \{0\}$, then $\calF$ is assumed to contain  
constant functions.  The deterministic variation is a semi-norm that is defined as the 
norm of the function  minus its typical value:
\begin{equation}  \label{FJH:eq:detvardef}
\FJHVAR^{\textup{D}}(f) := \FJHnorm{f- T(f)}_{\calF} \qquad \forall f \in \calF.
\end{equation} 

Let $\calM$ denote the vector space of signed measures for which integrands in 
$\calF$ 
have finite integrals: $\calM : = \bigl \{\text{signed measures } \eta : \FJHabs{\int_{\calX} 
f(\bsx) 
\, \eta(\D\bsx) } < \infty \ \ 
\forall f \in \calF \bigr \}$.
We assume that our integral of interest is defined, so $\nu \in \calM$.  Since function 
evaluation is bounded, $\calM$ includes $\calM_{\textup{S}}$ defined in 
\eqref{FJH:eq:sampmeaset} as well.  Define the subspace  
\begin{equation} \label{FJH:eq:Mperpdef}
\calM_\bot : = \begin{cases}
\left \{\eta \in \calM :  \eta(\calX) = 0
\right\}, & \{T(f) : f \in \calF\} \ne \{0\}, \\
\calM, & \{T(f) : f \in \calF\} = \{0\}.
\end{cases}
\end{equation} 
For example, if $\widehat{\nu}(\calX) = \nu(\calX)$, which is common, then $\nu 
-\widehat{\nu}$ is automatically in 
$\calM_\bot$.  However, in some situations $\widehat{\nu}(\calX) \ne \nu(\calX)$, as is 
noted in the discussion following \eqref{FJH:eq:RKHSvecmatDisc} below.
A semi-norm on $\calM_\bot$ is induced by the norm on $\calF$, which 
provides the definition of discrepancy:
\begin{equation} \label{FJH:eq:detdiscdef}
\FJHnorm{\eta}_{\calM_\bot}  : =\sup_{f \in \calF : f \ne 0} \frac{\displaystyle 
\FJHbiggabs{\int_{\calX} 
f(\bsx) \, \eta(\D \bsx)}}{\FJHnorm{f}_{\calF}}, \qquad \FJHDSC^\textup{D}(\nu - 
\widehat{\nu}) : = \FJHnorm{\nu - \widehat{\nu}}_{\calM_\bot}.
\end{equation}

Finally, define the confounding as 
\begin{equation} \label{FJH:eq:detconfdef}
\FJHCNF^\textup{D}(f,\nu - \widehat{\nu}): =  \begin{cases} \displaystyle 
\frac{\displaystyle\int_{\calX} f(\bsx) \, (\nu - \widehat{\nu})(\D 
\bsx)}{\FJHVAR^\textup{D}(f)\FJHDSC^\textup{D}(\nu - \widehat{\nu})},  & 
\FJHVAR^\textup{D}(f)\FJHDSC^\textup{D}(\nu - \widehat{\nu}) \ne 0, \\[2ex]
0, & \text{otherwise}.
\end{cases}
\end{equation}
The above definitions lead to the deterministic trio identity for cubature 
error.

\begin{theorem}[Deterministic Trio Error Identity]  \label{FJH:thm:dtrio} For the spaces 
of integrands and 
measures defined above, and for the above definitions of variation, discrepancy, and 
confounding, the following error identity holds for all $f \in \calF$ and $\nu - 
\widehat{\nu}  \in 
\calM_\bot$: 
\begin{equation} \tag{DTRIO} \label{FJH:eq:dtrio}
\mu - \widehat{\mu}  = \FJHCNF^\textup{D}(f,\nu - \widehat{\nu}) \, 
\FJHDSC^\textup{D}(\nu - \widehat{\nu}) \, 
\FJHVAR^{\textup{D}}(f).
\end{equation}
Moreover, $\FJHabs{\FJHCNF^\textup{D}(f,\nu - \widehat{\nu})} \le 1$. 
\end{theorem}
\begin{proof}  The proof of this identity follows from the definitions above.  It follows 
from \eqref{FJH:eq:INT} and \eqref{FJH:eq:AppxINT} that for all $f \in \calF$ and $\nu - 
\widehat{\nu}  \in \calM_\bot$, the cubature error can be written as a 
single integral:
	\begin{equation} \label{FJH:eq:err_as_int}
	\mu - \widehat{\mu}   =  \int_{\calX} f(\bsx) \, (\nu - \widehat{\nu})(\D \bsx).
	\end{equation}
	If $\FJHVAR^{\textup{D}}(f) = 0$, then $f = T(f)$, and the integral above vanishes by 
	the 
	definition of $\calM_{\bot}$.  If $\FJHDSC^\textup{D}(\nu - \widehat{\nu}) = 0$, then 
	the 
	integral 
	above 
	vanishes by \eqref{FJH:eq:detdiscdef}.  Thus, for $\FJHVAR^{\textup{D}}(f) 
	\FJHDSC^\textup{D}(\nu 
	- \widehat{\nu}) = 
	0$ 
	the trio identity holds. If $\FJHVAR^{\textup{D}}(f) \FJHDSC^\textup{D}(\nu - 
	\widehat{\nu}) \ne 
	0$, then the 
	trio 
	identity also holds by the definition of the confounding.
	
	Next, we bound the magnitude of the confounding for $\FJHVAR^{\textup{D}}(f) 
	\FJHDSC^\textup{D}(\nu - 
	\widehat{\nu}) \ne 0$: 
	\begin{align*}
	\FJHabs{\FJHCNF(f,\nu - \widehat{\nu})} & = 
		\frac{\biggl \lvert\displaystyle\int_{\calX} f(\bsx) \, (\nu - \widehat{\nu})(\D 
			\bsx) \biggr \rvert}{\FJHVAR^\textup{D}(f)\FJHDSC^\textup{D}(\nu - 
			\widehat{\nu})} \quad 
			\text{by 
			\eqref{FJH:eq:detconfdef}}\\
		& = \frac{\biggl \lvert\displaystyle\int_{\calX} [f(\bsx) - T(f)] \, (\nu - 
		\widehat{\nu})(\D 
			\bsx) \biggr \rvert}{\FJHnorm{f-T(f)}_{\calF}\FJHDSC^\textup{D}(\nu - 
			\widehat{\nu})} 
			\quad 
			\text{by 
			\eqref{FJH:eq:detvardef} and \eqref{FJH:eq:Mperpdef}} \\
		& \le 1 \quad \text{by \eqref{FJH:eq:detdiscdef}},
\end{align*}
since $\FJHVAR^{\textup{D}}(f) \ne 0$ and so $f - T(f) \ne 0$.
\end{proof}

Because $\FJHabs{\FJHCNF^\textup{D}(f,\nu - \widehat{\nu})} \le 1$, the deterministic 
trio 
identity implies a deterministic error bound:  $\FJHabs{\mu - \widehat{\mu}}  \le 
\FJHDSC^\textup{D}(\nu - 
\widehat{\nu}) 
\, \FJHVAR^{\textup{D}}(f)$.  However, there is value in keeping the confounding term as 
noted 
below in Lesson \ref{FJH:Lesson:Seven}.

The error in approximating the integral of $cf$ is $c$ times that for approximating the 
integral of $f$.  This is reflected in the fact that $\FJHVAR^\textup{D}(cf) = \lvert c \rvert 
\FJHVAR^\textup{D}(f)$ and 
$\FJHCNF(cf,\nu - \widehat{\nu}) = \textup{sign}(c) \FJHCNF(f,\nu - \widehat{\nu})$, 
while $\FJHDSC^\textup{D}(\nu - \widehat{\nu})$ does not 
depend on the integrand.

When $\calF$ is a Hilbert space with reproducing kernel $K$, the discrepancy has an 
explicit expression in terms of $K$.  The reproducing kernel is the unique function, $K: 
\calX \times \calX \to \R$ satisfying these two properties \cite[Sec.\ 1]{Aro50}:
\begin{equation*}
K(\cdot,\bst) \in \calF \quad  \text{and} \quad f(\bst) = 
\FJHip[\calF]{K(\cdot,\bst)}{f} 
\qquad \forall f \in \calF, \ \bst \in \calX.
\end{equation*}
The Riesz Representation Theorem implies that the representer of cubature error 
is 
\begin{equation*}
\eta_{\textup{err}}(\bst) = \FJHip[\calF]{K(\cdot,\bst)}{\eta_{\textup{err}}} = \int_{\calX} 
K(\bsx,\bst) \, (\nu - 
\widehat{\nu})(\D\bsx).
\end{equation*}
Thus, the deterministic trio identity for the reproducing kernel Hilbert space (RKHS) case 
is
\begin{equation*}
\mu - \widehat{\mu} =  \FJHip[\calF]{\eta_{\textup{err}}}{f} = 
\underbrace{\frac{\FJHip[\calF]{\eta_{\textup{err}}}{f}}{\FJHnorm{f - 
T(f)}_{\calF} \FJHnorm{\eta_{\textup{err}}}_{\calF}} }_{\FJHCNF^\textup{D}(f,\nu - 
\widehat{\nu})}\, 
\underbrace{\FJHnorm{\eta_{\textup{err}}}_{\calF}}_{\FJHDSC^{\textup{D}}(\nu - 
\widehat{\nu})} \, 
\underbrace{\FJHnorm{f - T(f)}_{\calF}}_{\FJHVAR^\textup{D}(f)}
\end{equation*}
provided that 
\begin{equation} \label{FJH:eq:Tnucond}
T(f) [\nu(\calX) - \widehat{\nu}(\calX)] = 0.
\end{equation}  
The squared discrepancy  takes the 
form \cite{Hic99a}
\begin{align*}
\nonumber
[\FJHDSC^{\textup{D}}(\nu - \widehat{\nu})]^2 & = 
\FJHnorm{\eta_{\textup{err}}}_{\calF}^2 = 
\FJHip[\calF]{\eta_{\textup{err}}}{\eta_{\textup{err}}} 
\\
\nonumber
& = \int_{\calX \times \calX} K(\bsx,\bst) \, (\nu - \widehat{\nu})(\D \bsx) \, (\nu - 
\widehat{\nu})(\D \bst) \\
\nonumber
& = \int_{\calX \times \calX} K(\bsx,\bst) \, \nu(\D \bsx) \, \nu (\D \bst)  \\
& \qquad \qquad - 2 \sum_{i=1}^n w_i 
\int_{\calX} K(\bsx_i,\bst) \, \nu(\D \bst)+ \sum_{i,j = 1}^n w_iw_jK(\bsx_i,\bsx_j).
\end{align*}
Assuming that the single integral and double integral of the reproducing kernel can be 
evaluated analytically, the computational 
cost to evaluate the discrepancy is $\calO(n^2)$ unless the kernel 
has 
a special form that speeds up the calculation of the double sum. 

\begin{FJHLesson} \FJHLessonTwoHalf \end{FJHLesson}

In the RKHS case, the confounding corresponds to 
the cosine of the angle between  $f - T(f)$ and the cubature error representer, 
$\eta_{\textup{err}}$.  This cosine is no greater than one in magnitude, as expected.

The square deterministic discrepancy for an RKHS may be expressed in terms of vectors 
and matrices:
\begin{subequations} \label{FJH:eq:RKHSvecmatDisc}
	\begin{gather} 
   \bsw = \bigl(w_i 
	\bigr)_{i=1}^n, \qquad  k_0 = \int_{\calX} 
	K(\bsx,\bst) \,\nu(\D \bsx) \,\nu(\D \bst), \\
	\bsk = 
	\biggl( 
	\int_{\calX} 
	K(\bsx_i,\bst) \,\nu(\D \bst) \biggr)_{i=1}^n, \qquad
	\mathsf{K}  = \bigl( K(\bsx_i,\bsx_j) \bigr)_{i,j=1}^n, \\
	[\FJHDSC^{\textup{D}}(\nu - \widehat{\nu})]^2 =  k_0 - 2 \bsk^T \bsw + \bsw^T 
	\mathsf{K} \bsw.
	\end{gather}
\end{subequations}
Given fixed data sites, the optimal cubature weights to minimize the discrepancy are  
$\bsw = \mathsf{K}^{-1} \bsk$.
If $\bsone^T\mathsf{K}^{-1} \bsk = 1$, which is possible but not automatic, then 
$\widehat{\nu}(\calX) = \nu(\calX) = 1$ for these optimal weights, and 
\eqref{FJH:eq:dtrio} holds for general 
$T$.  Otherwise, one must define $T(f) = 0$ for all $f \in \calF$ to satisfy condition 
\eqref{FJH:eq:Tnucond} for these optimal cubature weights.

A particular example of this RKHS setting corresponds to the uniform probability measure 
$\nu$ on the $d$-dimensional unit cube, $\calX = 
[0,1]^d$, and the reproducing kernel defined by \cite{Hic97a}
\begin{equation} \label{FJH:eq:L2Kdef}
K(\bsx,\bst) =\prod_{k = 1}^d [2 - \max(x_k,t_k)].
\end{equation}
 In this example, $T(f) = f(\bsone)$, and the 
variation is 
\begin{equation*}
\FJHVAR^\textup{D}(f)  = \FJHnorm{f-f(\bsone)}_\calF = \Bigl \lVert \bigl ( 
\FJHnorm{\partial^\fraku 
f}_{L^2}\bigr 
)_{\emptyset \subsetneq \fraku \subseteq 1:d} \Bigr \rVert_{\ell^2} , \qquad 
\partial^\fraku f : = \frac{\partial^{\lvert\fraku\rvert} f}{\partial \bsx_\fraku} \biggr 
\rvert_{\bsx_{\bar{\fraku}} = \bsone}.
\end{equation*}
Here $1\!:\!d$ means  $= \{1, \ldots, d\}$, $x_\fraku$ means $(x_k)_{k \in \fraku}$, and 
$\bar{\fraku}$ denotes the complement of $\fraku$.  
The square discrepancy for the equally weighted case with $w_1 = \cdots = w_n = 1/n$ 
is
\begin{align}
\nonumber
[\FJHDSC^\textup{D}(\nu - \widehat{\nu})]^2  &= \left(\frac 43\right)^d - \frac{2}n 
\sum_{i=1}^n 
\prod_{k=1}^d 
\left (\frac{3 - x_{ik}^2}{2} \right) + \frac{1}{n^2}\sum_{i,j=1}^n \prod_{k = 1}^d [2 - 
\max(x_{ik},x_{jk})] 
\\ &= \Bigl \lVert \bigl ( \FJHnorm{\nu([\bszero,\cdot_\fraku]) - 
	\widehat{\nu}([\bszero,\cdot_\fraku])}_{L^2}\bigr )_{\emptyset \subsetneq \fraku 
	\subseteq 1:d} 
	\Bigr 
	\rVert_{\ell^2}. \label{FJH:eq:L2disc}
\end{align}

This discrepancy has a geometric interpretation:  $\nu([\bszero,\bsx_\fraku])$ 
corresponds to the \emph{volume} of the $\lvert \fraku \rvert$-dimensional box 
$[\bszero,\bsx_\fraku]$, and $\widehat{\nu}([\bszero,\bsx_\fraku])$ corresponds to the 
\emph{proportion} 
of data sites lying in the box $[\bszero,\bsx_\fraku]$. The discrepancy in 
\eqref{FJH:eq:L2disc}, which is 
called the $L^2$-discrepancy, depends on difference between this volume and this 
proportion 
for all $\bsx \in [0,1]^d$ and  for all $\emptyset \subsetneq \fraku \subseteq 1\!:\!d$. 

If the data sites  $\bsx_{1}, \ldots, \bsx_{n}$ are chosen to be IID with probability measure 
$\nu$, and $w_1 = \cdots = w_n = 1/n$, then the mean square discrepancy for 
the RKHS case is
\begin{equation*}
\bbE \bigl\{[\FJHDSC^{\textup{D}}(\nu - \widehat{\nu})]^2 \bigr \}  = \frac 1 n \left [ 
\int_{\calX} 
K(\bsx,\bsx) \, 
\nu(\D 
\bsx) - 
\int_{\calX \times \calX} K(\bsx,\bst) \, \nu(\D \bsx) \, \nu (\D \bst) \right ].
\end{equation*}
For the $L^2$-discrepancy in \eqref{FJH:eq:L2disc} this becomes 
\begin{equation} \label{FJH:eq:IIDL2disc}
\bbE \bigl\{[\FJHDSC^{\textup{D}}(\nu - \widehat{\nu})]^2 \bigr \} = \frac 1 n \left [ 
\left(\frac 32\right)^d - 
\left(\frac 43\right)^d \right ].
\end{equation}

Quasi-Monte Carlo methods generally employ sampling measures of the form  
$\widehat{\nu} = 
n^{-1}\sum_{i=1}^n \delta_{\bsx_i}$, but choose the data sites $\{\bsx_i\}_{i=1}^n$ to be 
better than IID in the sense of discrepancy.  For 
integration over $\calX = [0,1]^d$ with 
respect to the uniform measure, these \emph{low discrepancy} data sites may come from
\begin{itemize} 
\item a digital sequence \cite{DicPil10a}, such as that proposed by Sobol' \cite{Sob67}, 
Faure \cite{Fau82}, Niederreiter \cite{Nie88}, or Niederreiter and Xing \cite{NieXin96}, or 
\item a sequence of node sets of an integration lattice \cite{SloJoe94}.  
\end{itemize}
The 
constructions of such sets are described in the references above and L'Ecuyer's 
tutorial in this volume.  The $L^2$-discrepancy defined in \eqref{FJH:eq:L2disc} and its 
relatives are 
$\calO(n^{-1+\epsilon})$ as $n \to \infty$ for any positive $\epsilon$ for these low 
discrepancy data sites \cite{Nie92}.  

Fig.\ \ref{FJH:fig:plotsdiffpts} displays examples of IID and randomized low discrepancy 
data sites.  
Fig.\ \ref{FJH:fig:unwtdiscdiffpts} shows the rates of decay for the $L^2$-discrepancy 
for various dimensions.  The scaled discrepancy is the empirically computed root mean 
square discrepancy divided by its value for $n=1$.  Although the 
decay for the low discrepancy points is 
$\calO(n^{-1+\epsilon})$ for large enough $n$, the decay in Fig.\ 
\ref{FJH:fig:unwtdiscdiffpts} resembles $\calO(n^{-1/2})$ for large dimensions and 
modest $n$.    The scaled discrepancy for IID samples in Fig.\ 
\ref{FJH:fig:unwtdiscdiffpts} does not 
exhibit a dimension dependence because 
it is masked by the 
scaling.  The dimension 
dependence of the convergence of the discrepancy to zero is addressed later in Sec.\  
\ref{FJH:sec:Tract}. 

\begin{FJHLesson}
	\FJHLessonTwo
\end{FJHLesson}

\begin{figure}
	\centering
	\includegraphics[height=0.7\FJHfigheight]{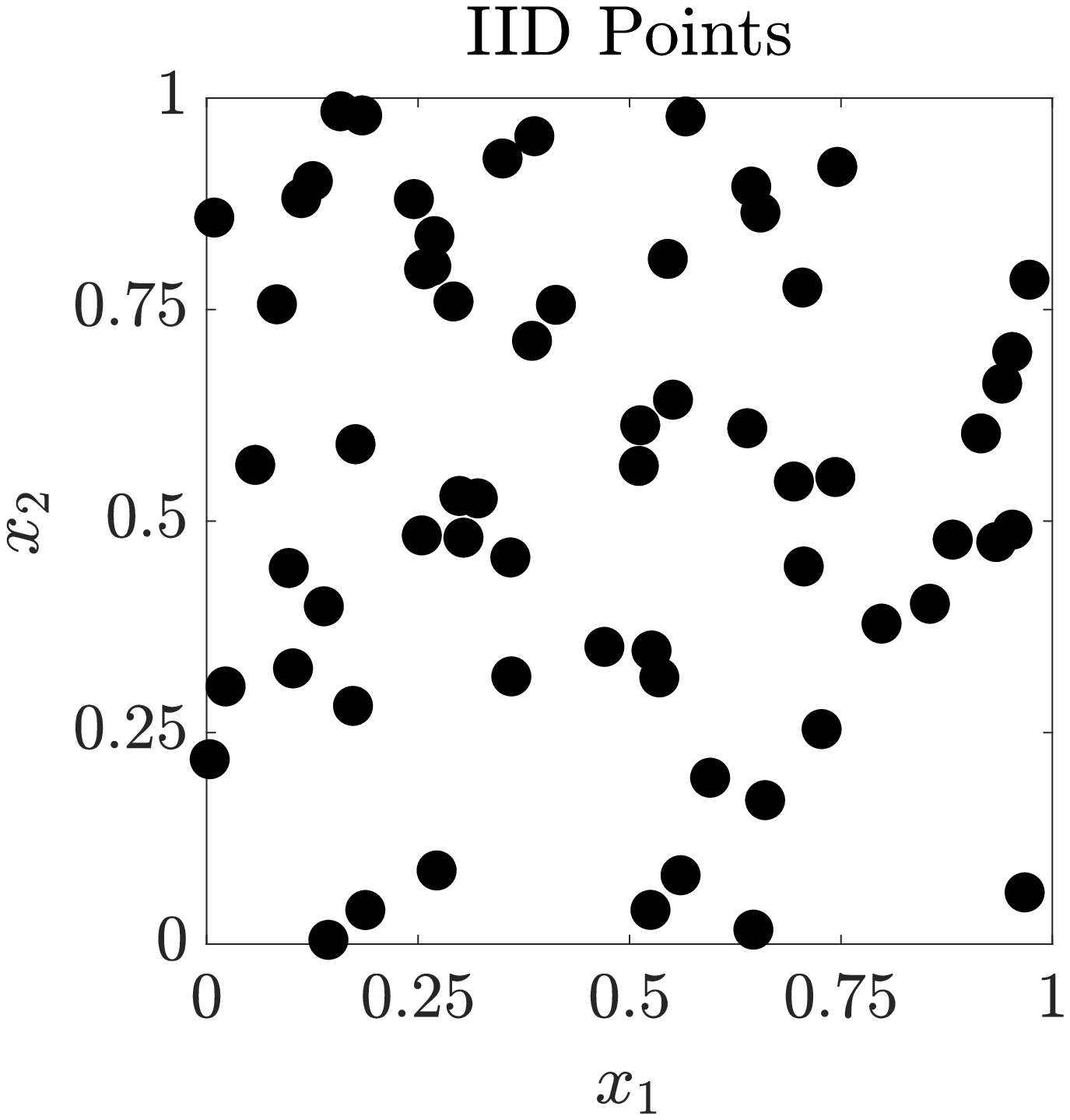} \ 
	\
	\includegraphics[height=0.7\FJHfigheight]{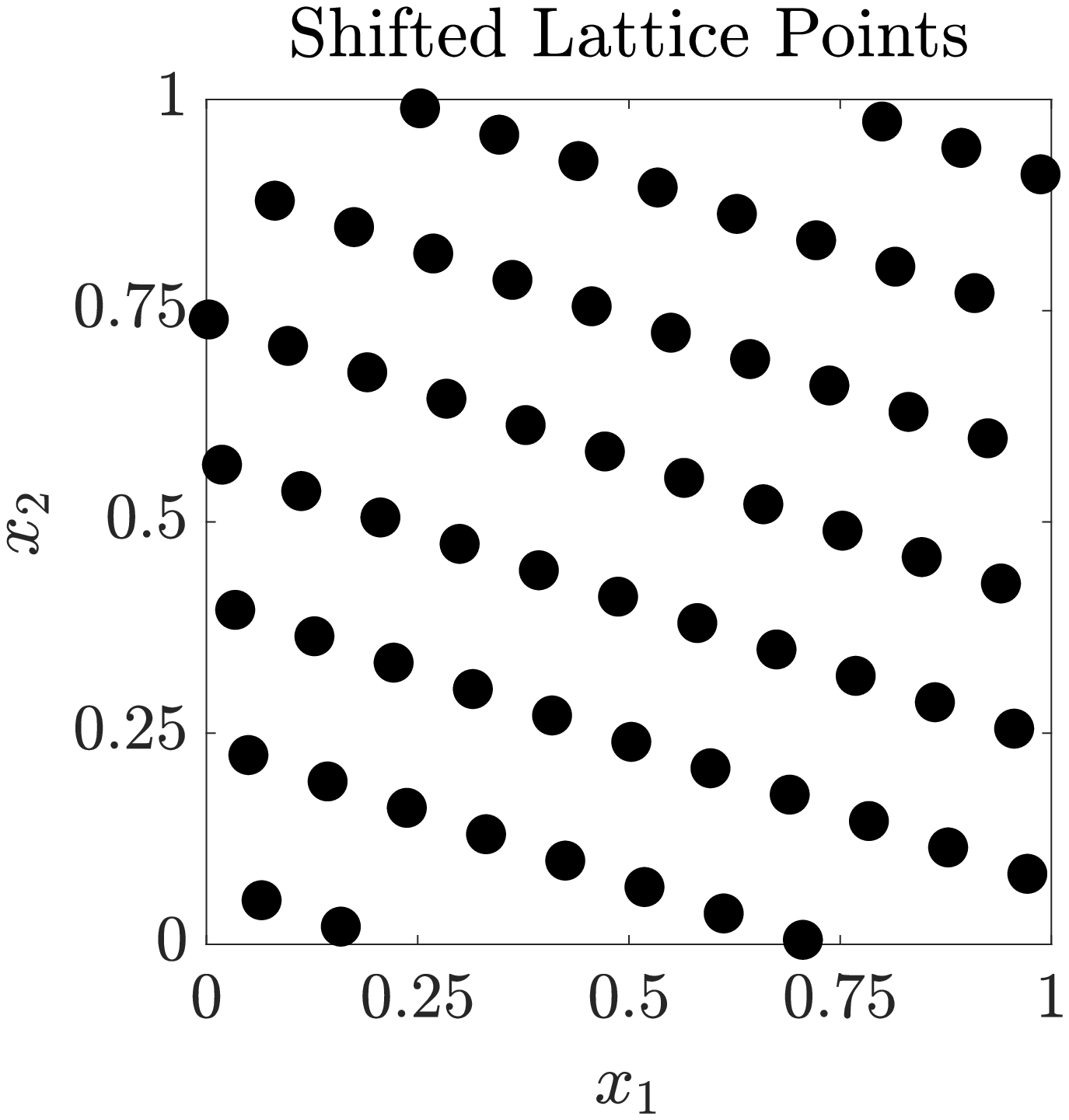} 
	\ \
	\includegraphics[height=0.7\FJHfigheight]{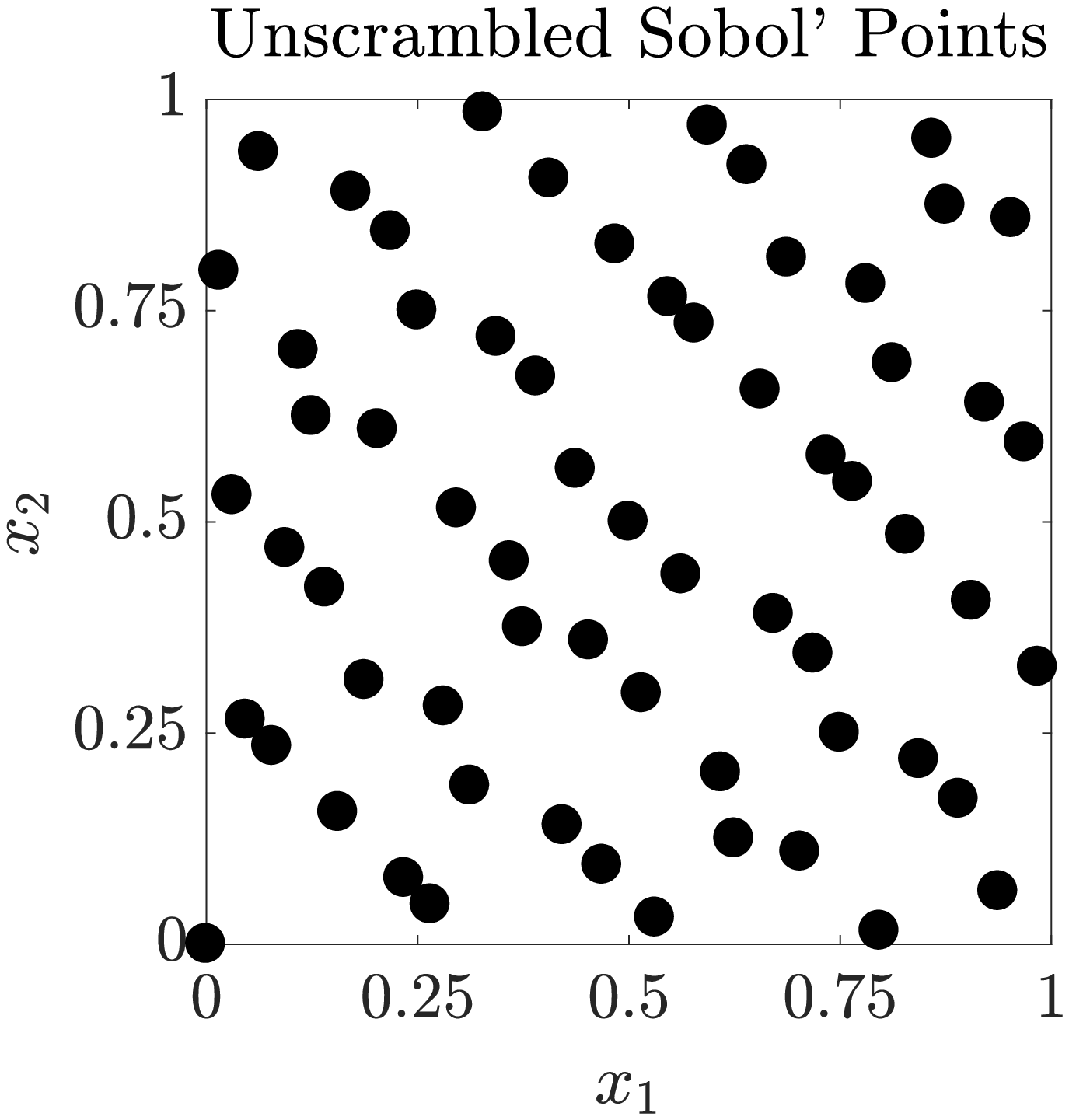} \ \
	\includegraphics[height=0.7\FJHfigheight]{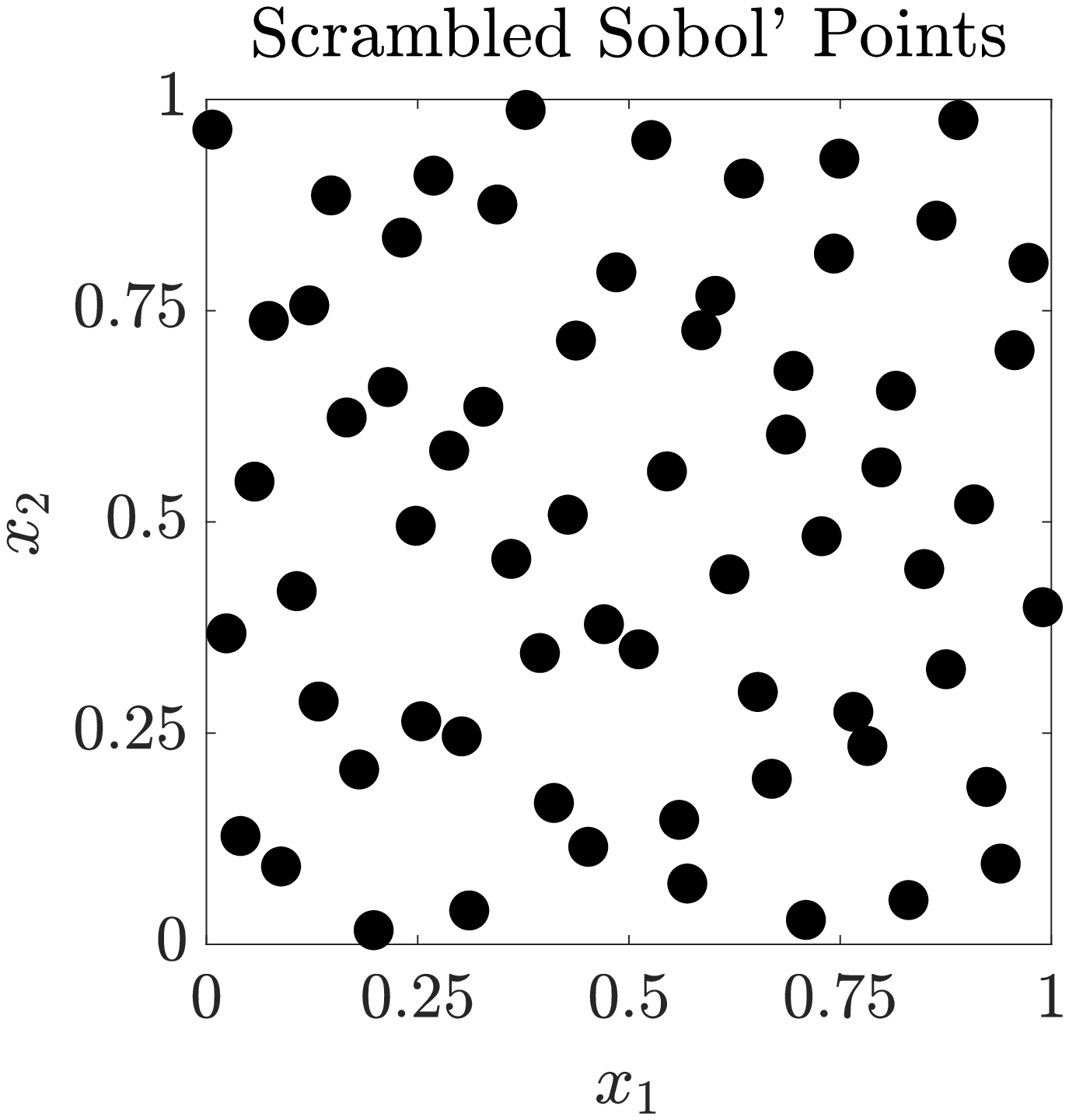}
	\caption{IID points and three examples of low discrepancy points 
	\label{FJH:fig:plotsdiffpts}}
\end{figure}

\begin{figure}
	\centering
	  \includegraphics[height=\FJHfigheight]{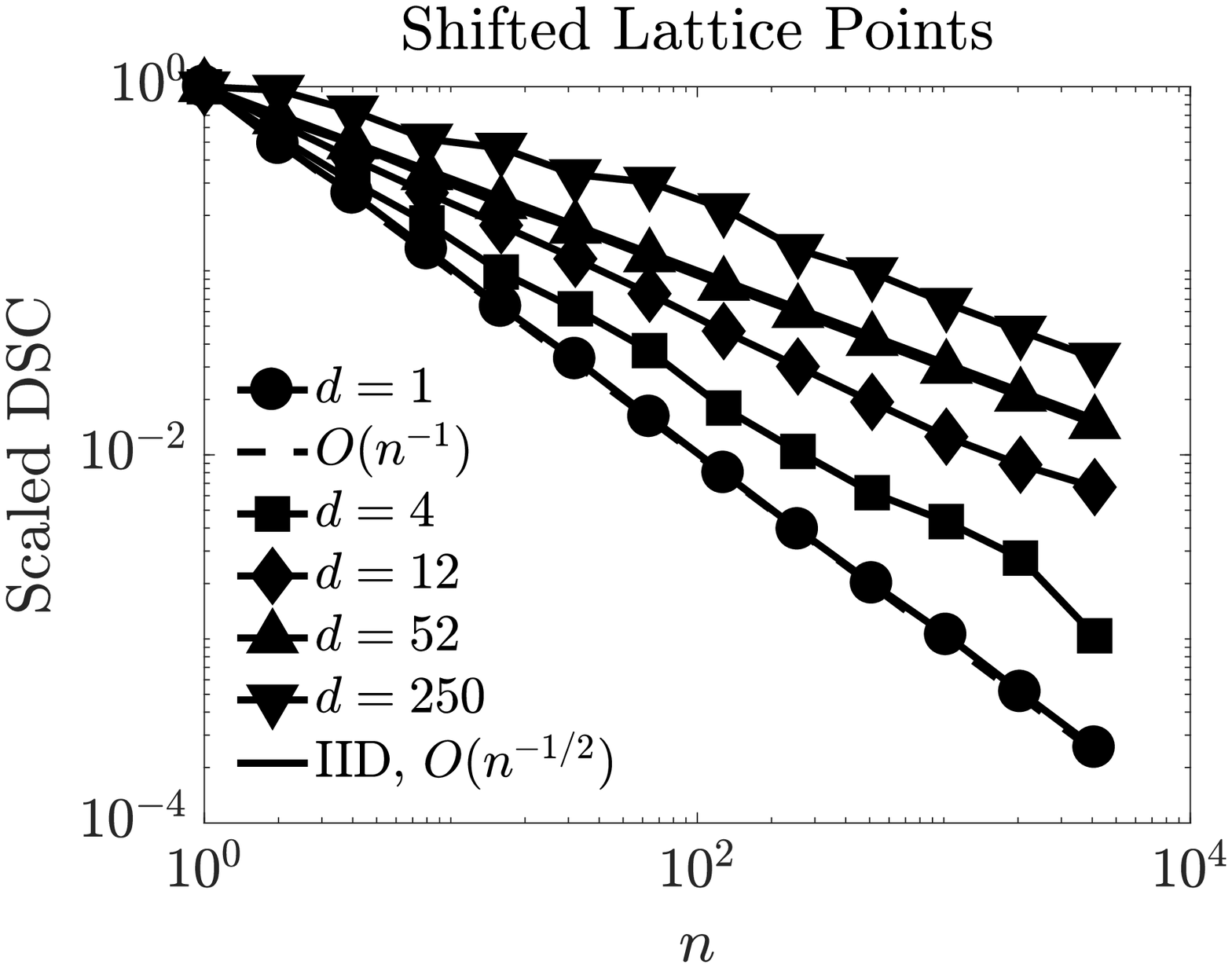}   
	  \qquad 
	  \includegraphics[height=\FJHfigheight]{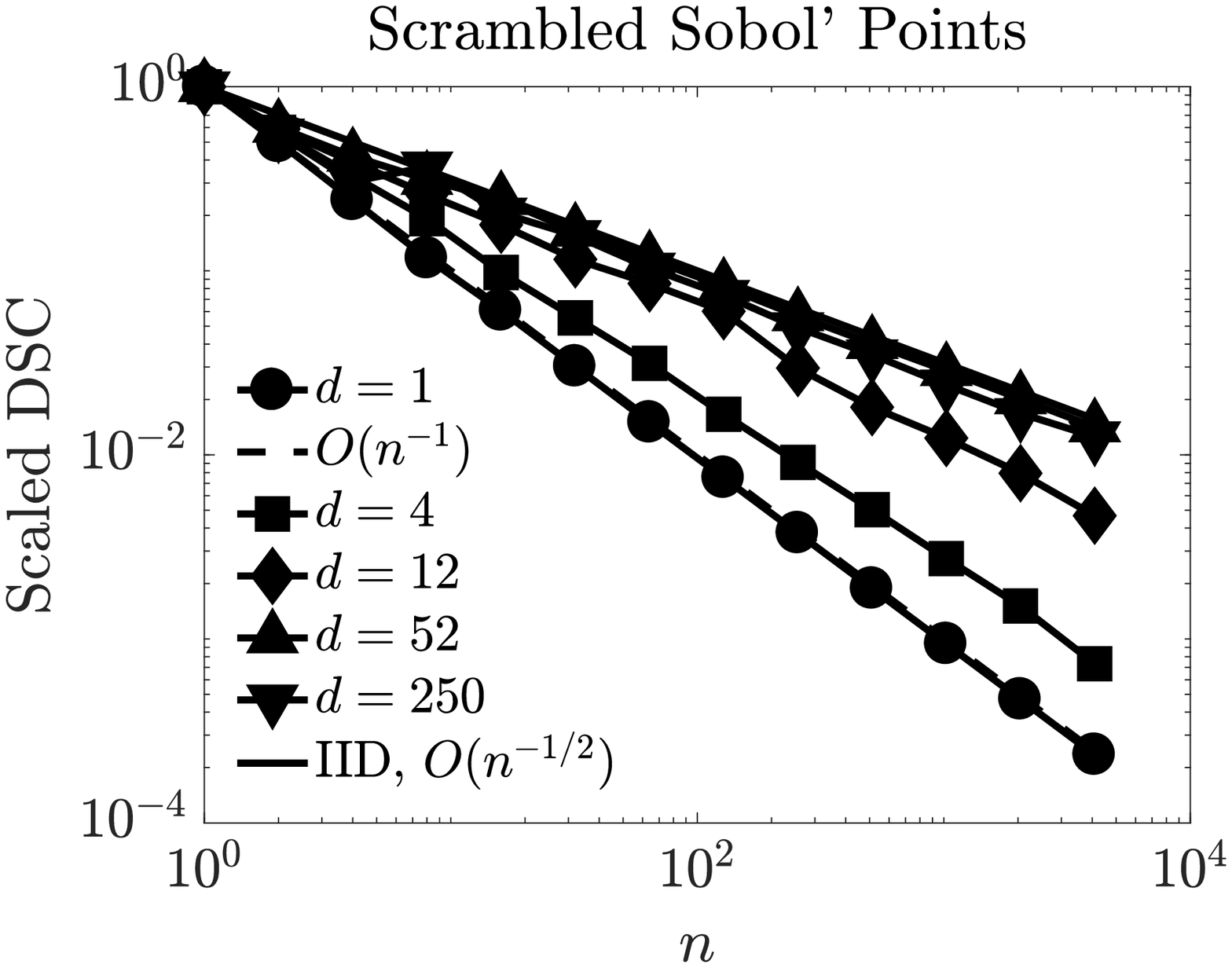} 
	\caption{The root mean square $L^2$-discrepancies given by \eqref{FJH:eq:L2disc} 
	for randomly shifted 
	lattice sequence nodesets and randomly scrambled and shifted Sobol' sequences 
	points for a variety of dimensions.
		\label{FJH:fig:unwtdiscdiffpts}}
\end{figure}

No sampling scheme can produce a faster convergence rate than 
$\calO(n^{-1})$ for the $L^2$-discrepancy.  This is due to the limited 
smoothness of the reproducing kernel defined in \eqref{FJH:eq:L2Kdef} and the 
corresponding limited smoothness of the corresponding Hilbert space of integrands.

\section{A Randomized Trio Identity for Cubature Error} \label{FJH:sec:rndtrio}
For the randomized version of the trio identity, we again assume that the integrands lie in 
a Banach space, $(\calF,\FJHnorm{\cdot}_{\calF})$.  This space is required to contain 
constant functions if $\{T(f) : f \in \calF\} \ne \{0\}$.  We assume
that $\int_{\calX} f(\bsx) \, \nu(\D\bsx)$ is defined  for all $f \in \calF$, however, we  
do not require 
function evaluation to be a bounded linear functional on $\calF$.  The definitions of the 
bounded linear functional $T$ and the variation in the deterministic case in  
\eqref{FJH:eq:detvardef} apply here as well.

Now endow the vector space of all sampling measures, $\calM_{\textup{S}}$, with a 
probability distribution.  This means that the placement of the data sites, the number of 
data sites, and the choice of the weights may all be random.  We require  
the following two conditions:
\begin{gather}
\nonumber
\bbE_{\widehat{\nu}} \FJHbiggabs{\int_{\calX} f(\bsx) \,  \widehat{\nu}(\D \bsx)}^2 < \infty 
\qquad \forall f \in 
\calF, 
\\
\label{FJH:eq:randMcondB}
\{T(f) : f \in \calF\} = \{0\}  \qquad \text{or} \qquad  \widehat{\nu}(\calX) = \nu(\calX) 
\text{ almost surely.} 
\end{gather}
The first condition implies that $\int_{\calX} f(\bsx) \,  \widehat{\nu}(\D \bsx)$ exists 
almost surely for every $f \in \calF$.  

The randomized discrepancy is  defined as the worst  normalized root mean squared 
error:
\begin{equation} \label{FJH:eq:rnddiscdef}
\FJHDSC^\textup{R}(\nu - \widehat{\nu}) : =\sup_{f \in \calF : f \ne 0} \frac{\displaystyle 
\sqrt{\bbE_{\widehat{\nu}}
\FJHbiggabs{\int_{\calX} 
		f(\bsx) \, (\nu - \widehat{\nu})(\D \bsx)}^2}}{\FJHnorm{f}_{\calF}}.
\end{equation}
The randomized discrepancy does not depend on the particular instance of the 
sampling measure but on the distribution of the sampling measure. 

Finally, define the confounding as 
\begin{equation} \label{FJH:eq:rndconfdef}
\FJHCNF^\textup{R}(f,\nu - \widehat{\nu}): =  \begin{cases} \displaystyle 
\frac{\displaystyle\int_{\calX} f(\bsx) \, (\nu - \widehat{\nu})(\D 
	\bsx)}{\FJHVAR^\textup{D}(f)\FJHDSC^\textup{R}(\nu - \widehat{\nu})},  & 
\FJHVAR^\textup{D}(f)\FJHDSC^\textup{R}(\nu - \widehat{\nu}) \ne 0, \\[2ex]
0, & \text{otherwise}.
\end{cases}
\end{equation}
Here, the confounding \emph{does} depend on the particular instance of the sampling 
measure.  The above definitions allow us to establish the randomized trio identity for 
cubature error.

\begin{theorem}[Randomized Trio Error Identity]  \label{FJH:thm:rtrio} For the spaces of 
integrands and 
	measures defined above, and for the above definitions of variation, discrepancy, and 
	confounding, the following error identity holds for all $f \in \calF$ and $\widehat{\nu}  
	\in 
	\calM_{\textup{S}}$: 
	\begin{equation} \tag{RTRIO} \label{FJH:eq:rtrio}
	\mu - \widehat{\mu}  = \FJHCNF^\textup{R}(f,\nu - \widehat{\nu}) \, 
	\FJHDSC^\textup{R}(\nu - \widehat{\nu}) \, 
	\FJHVAR^{\textup{D}}(f) \quad 
	\text{almost surely}.
	\end{equation}
	Moreover, $\bbE_{\widehat{\nu}} \FJHabs{\FJHCNF^\textup{R}(f,\nu - 
	\widehat{\nu})}^2  \le 1$ for all $f \in 
	\calF$. 
\end{theorem}
\begin{proof}  For all $f \in \calF$ and $\widehat{\nu}  \in \calM_{\textup{S}}$, the error 
can be 
written as the single integral in \eqref{FJH:eq:err_as_int} almost surely. 	If 
$\FJHVAR^{\textup{D}}(f) 
= 0$, then $f = T(f)$, and $\mu - \widehat{\mu}$ vanishes almost surely by
\eqref{FJH:eq:randMcondB}.  If 
$\FJHDSC^\textup{R}(\nu - \widehat{\nu}) = 0$, then $\mu - \widehat{\mu}$
	vanishes almost surely by \eqref{FJH:eq:rnddiscdef}.  Thus, for 
	$\FJHVAR^{\textup{D}}(f) 
	\FJHDSC^\textup{R}(\nu - \widehat{\nu}) = 
	0$ 
	the trio identity holds. If $\FJHVAR^{\textup{D}}(f) \FJHDSC^\textup{R}(\nu - 
	\widehat{\nu}) \ne 
	0$, 
	then the 
	trio 
	identity also holds by the definition of the confounding.
	
	Next, we analyze the magnitude of the confounding for $\FJHVAR^{\textup{D}}(f) 
	\FJHDSC^\textup{D}(\nu - 
	\widehat{\nu}) \ne 0$: 
	\begin{align*}
	\bbE \FJHabs{\FJHCNF^\textup{R}(f,\nu - \widehat{\nu})}^2 & = 
	\frac{\bbE \biggl \lvert\displaystyle\int_{\calX} f(\bsx) \, (\nu - \widehat{\nu})(\D 
		\bsx) \biggr \rvert^2 }{[\FJHVAR^\textup{D}(f)\FJHDSC^\textup{D}(\nu - 
		\widehat{\nu})]^2} 
		\quad \text{by 
		\eqref{FJH:eq:rndconfdef}}\\
	& = \frac{\bbE \biggl \lvert\displaystyle\int_{\calX} [f(\bsx) - T(f)] \, (\nu - 
	\widehat{\nu})(\D 
		\bsx) \biggr \rvert^2}{[\FJHnorm{f-T(f)}_{\calF}\FJHDSC^\textup{D}(\nu - 
		\widehat{\nu})]^2} 
		\quad 
		\text{by 
		\eqref{FJH:eq:detvardef} and \eqref{FJH:eq:randMcondB}} \\
	& \le 1 \quad \text{by \eqref{FJH:eq:rnddiscdef}},
	\end{align*}
	since $\FJHVAR^{\textup{D}}(f) \ne 0$ and so $f - T(f) \ne 0$.
\end{proof}

Consider simple Monte Carlo, where the approximation to the integral is an equally 
weighted average using IID sampling $\bsx_1, \bsx_2, \ldots \sim 
\nu$. Let the sample size be fixed at $n$.
Let $\calF = L^{2,\nu}(\calX)$, the space of functions that are square integrable with 
respect to 
the measure $\nu$, and let $T(f)$ be the mean of $f$.  Then the variation of $f$ is just 
its standard 
deviation, $\textup{std}(f)$.  The randomized discrepancy is 
$1/\sqrt{n}$.  The randomized 
confounding is 
\begin{align*}
\FJHCNF^\textup{R}(f,\nu - \widehat{\nu}) = \frac{-1}{ \sqrt{n}\, \textup{std}(f)}    
\sum_{i=1}^n 
[f(\bsx_i) - 
\mu].
\end{align*}

Unlike the deterministic setting, there is no simple expression  for the randomized 
discrepancy under general sampling measures and RKHSs.  The 
randomized discrepancy can sometimes be conveniently calculated or bounded for 
spaces of integrands that are represented by series expansions, and the randomized 
sampling measures for the bases of these expansions have special properties 
\cite{HeiHicYue02a, HicWoz00a}.

It is instructive to contrast the variation, discrepancy, and confounding in the 
deterministic and randomized settings.   For some integrand, $f$, and some sampling 
measure, $\widehat{\nu}$, satisfying the 
conditions defining both \eqref{FJH:eq:dtrio}  and \eqref{FJH:eq:rtrio}:
\begin{itemize}
	\item the variation in both settings is the same,
	
	\item the randomized discrepancy 
	must be \emph{no greater} than the deterministic discrepancy by definition, and thus
	
	\item the randomized confounding must be \emph{no less} than the deterministic 
	confounding.  
\end{itemize}
The deterministic confounding is never greater 
than one in magnitude.  By contrast, the randomized confounding may be arbitrarily 
large.  However, Markov's 
inequality implies that it may be larger than $1/\sqrt{\alpha}$ with  probability no greater 
than $\alpha$.  The next section illustrates the differences in the 
deterministic and randomized trio identities. 

\section{Multivariate Gaussian Probabilities} \label{FJH:sec:Gauss}
Consider the $d$-variate integral corresponding to the probability of a 
$\calN(\bszero,\mathsf{\Sigma})$ random variable lying inside the box $[\bsa,\bsb]$:
\begin{equation} \label{FJH:eq:MVN}
\mu = \int_{[\bsa,\bsb]} \frac{\exp\bigl(- \frac 12 \bsz^T \mathsf{\Sigma}^{-1} \bsz 
\bigr)}{\sqrt{(2 \pi)^d \det(\mathsf{\Sigma})}} \, \D \bsz = \int_{[0,1]^{d-1}} 
f_{\textup{Genz}}(\bsx) \, \D 
\bsx,
\end{equation}
where $\mathsf{\Sigma} = \mathsf{L}\mathsf{L}^T$ is the Cholesky decomposition of 
the covariance matrix, $\mathsf{L} = \bigl(l_{jk}\bigr)_{j,k=1}^d$, is a lower triangular 
matrix, and
\begin{align*}
\nonumber
\alpha_1 & = \Phi(a_1), \qquad \beta_1  = \Phi(b_1), \\
\nonumber
\alpha_j(x_1, \ldots, x_{j-1}) &= \Phi\left(\frac{1}{l_{jj}} \left(a_j - \sum_{k=1}^{j-1} 
l_{jk}\Phi^{-1}(\alpha_k + x_k(\beta_k-\alpha_k))\right)\right), \ j =2, \ldots, d,\\
\nonumber
\beta_j(x_1, \ldots, x_{j-1}) &= \Phi\left(\frac{1}{l_{jj}} \left(b_j - \sum_{k=1}^{j-1} 
l_{jk}\Phi^{-1}(\alpha_k + x_k(\beta_k-\alpha_k))\right)\right), \ j =2, \ldots, d, \\
f_{\textup{Genz}}(\bsx) &= \prod_{j=1}^{d} [\beta_j(\bsx) - \alpha_j(\bsx)]. 
\end{align*}
Here, $\Phi$ represents the cumulative distribution function for a standard normal 
random variable.  Genz \cite{Gen93} developed this clever transformation of variables 
above.  Not only is the 
dimension decreased by one, but the integrand is typically made less peaky and more 
favorable to cubature methods.

The left plot of Fig.\ \ref{FJH:fig:MVNfig} shows the absolute errors in computing the 
multivariate 
Gaussian probability via the Genz transformation for 
\[
   \bsa  = \begin{pmatrix}
   -6 \\ -2 \\ -2
   \end{pmatrix}, \quad
      \bsb  = \begin{pmatrix}
   5 \\ 2 \\ 1
   \end{pmatrix}, \quad
   \mathsf{\Sigma} = \begin{pmatrix} 16 & 4 & 4 \\ 4 &  2 &  1.5 \\
  4 & 1.5 &  1.3125 \end{pmatrix}, \quad
   \mathsf{L} = \begin{pmatrix} 4 & 0 & 0 \\ 1 &  1 &  0 \\
1 & 0.5 &  0.25 \end{pmatrix}, 
\]
by IID sampling, unscrambled Sobol' sampling, and scrambled Sobol' sampling   
\cite{Owe95, Owe96, Owe97}.
Multiple random scramblings of a very large scrambled Sobol' set were used to infer that 
$\mu \approx 0.6763373243578$.   For the two randomized sampling measures $100$ 
replications were taken.  The marker denotes the median error and the top of the stem 
extending above the marker denotes the  $90\%$ quantile of the error.

Empirically, the error for scrambled Sobol' sampling appears to be  tending towards a 
convergence rate 
of $\calO(n^{-2})$.  This is a puzzle.  It is unknown why this should be or whether this 
effect is only transient.  In the discussion below we assume the expected rate of 
$\calO(n^{-3/2 + \epsilon})$.

\begin{figure}
	\centering
	\includegraphics[height = \FJHfigheight]{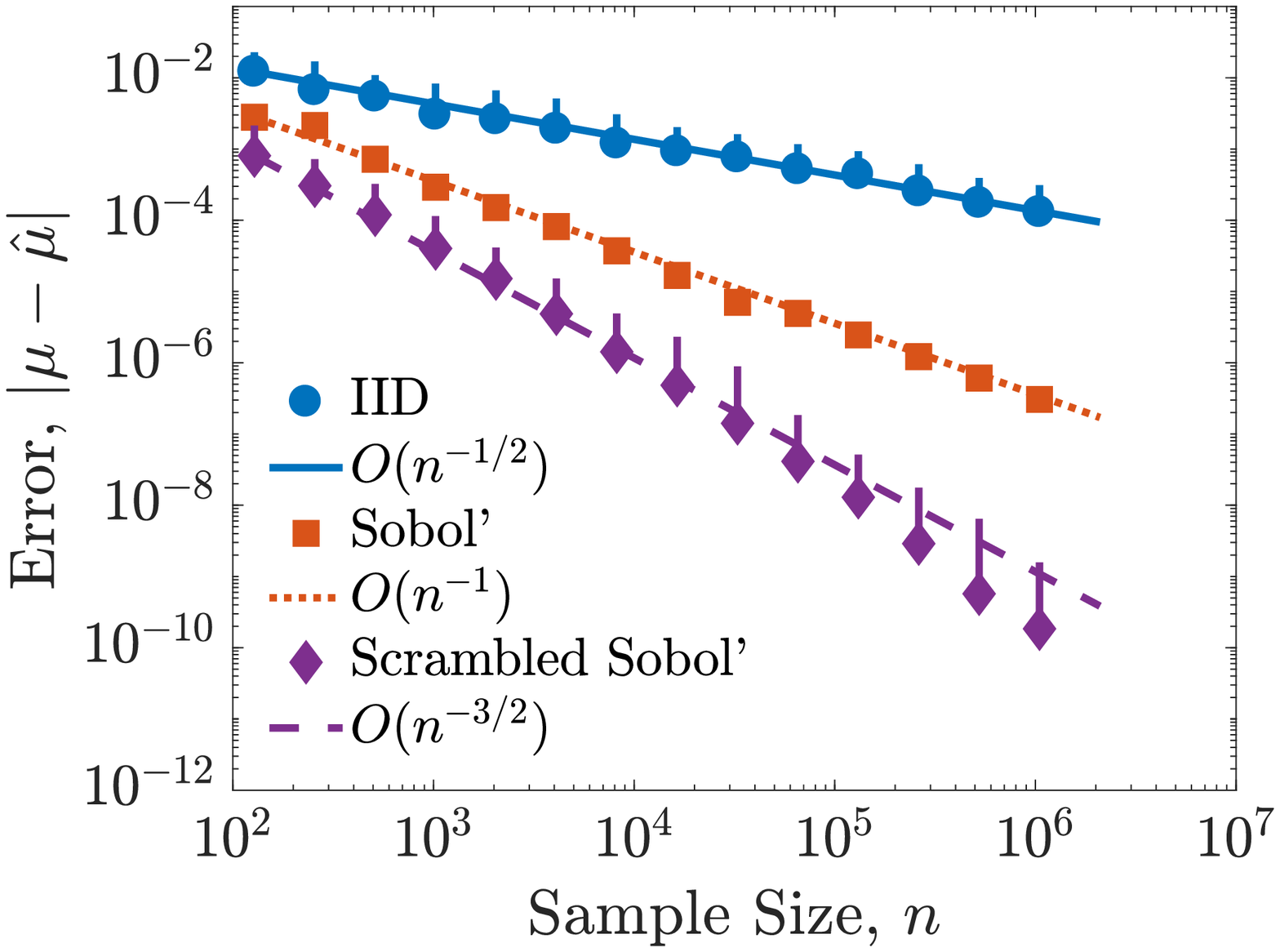} 
	\qquad 
	\includegraphics[height = \FJHfigheight]{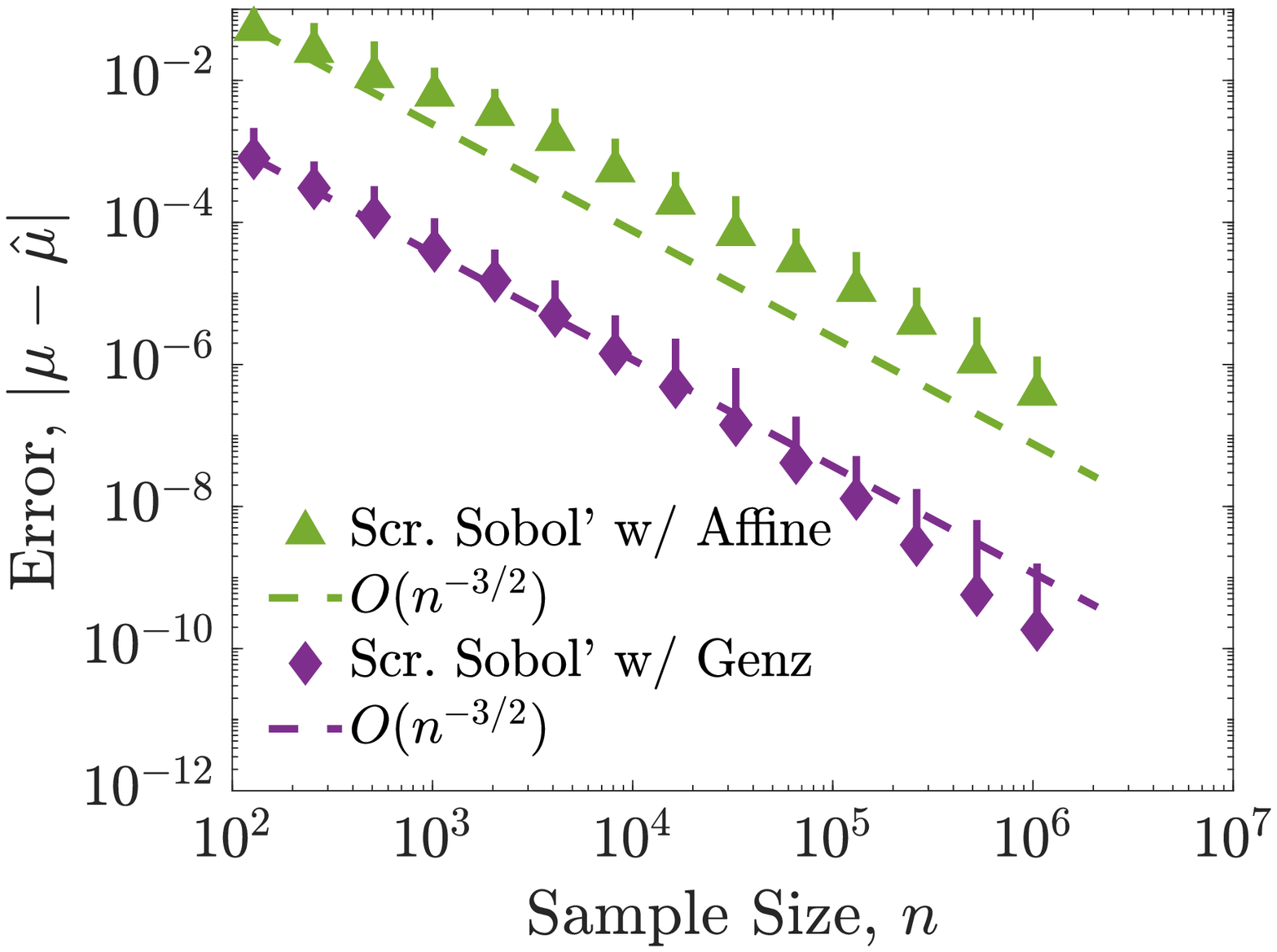}
	\caption{The error of an example of the multivariate Gaussian probability in 
	\eqref{FJH:eq:MVN}.  The left side shows the result of Genz's transformation and 
	different sampling measures.  The right side shows scrambled Sobol' sampling
	using different transformations.
	\label{FJH:fig:MVNfig}}
\end{figure}

The orders of the discrepancy and confounding in Table \ref{FJH:tab:Conf} explain the 
rates of decay of the error and the 
benefits of randomization.  Note that in all cases 
\[
\mu - \widehat{\mu} \text{ decay rate} = \FJHCNF \text{ decay/growth rate} \times 
\FJHDSC
\text{ decay rate}.
\]
We consider  equally 
weighted cubature rules for two  kinds 
of random sampling measures, IID and scrambled Sobol', and for both the deterministic 
and randomized settings.  Here, $\calF$ is assumed to be the 
RKHS used to define the $L^2$-discrepancy.

\begin{table}
	\caption{Confounding orders for deterministic randomized settings and two
		different sets of equi-weighted random
		sampling measures. Sufficient smoothness of the integrand is assumed.
		The order of the error equals the order of the discrepancy times the 
		order of the confounding. \label{FJH:tab:Conf}}
	\begin{equation*}
	\begin{array}{r@{\ \ }c@{\quad}r@{\ \ }|c@{\ \ }c@{\quad}c@{\ \ }c@{\ \ }c}
    & \textup{RMS} &&&&\text{Typical} &  \textup{Unscr.} 
	&\text{Typical}\\
	\textbf{Deterministic setting} & L^2\text{-}\FJHDSC^\textup{D} & \widehat{\nu} &&  
	\textup{Worst} &  
	\textup{IID} & \textup{Sobol'}&  \textup{Scr.\ Sobol'}\\
	\hline \multicolumn{3}{c|}{} &  \\ [-2ex]
	\textup{IID Sampling} & \calO(n^{-1/2}) & \FJHCNF^\textup{D} &&  & \calO(1) 
	&\calO(n^{-1/2+\epsilon}) &\calO(n^{-1+\epsilon}) \\
	\textup{Scr.\ Sobol' Sampling} & \calO(n^{-1+\epsilon}) & \FJHCNF^\textup{D} 
	&& & &\calO(1)&\calO(n^{-1/2+\epsilon}) \\
	&& \mu - \widehat{\mu} && \calO(1) & \calO(n^{-1/2}) & 
	\calO(n^{-1+\epsilon}) & \calO(n^{-3/2+\epsilon})\\
\textbf{Randomized setting} & \FJHDSC^\textup{R} & \\
\cline{1-3}
\multicolumn{3}{c|}{} &  \\ [-2ex]
\textup{IID Sampling} & \calO(n^{-1/2}) & \FJHCNF^\textup{R} && \calO(n^{1/2})  & 
\calO(1) 
&\calO(n^{-1/2+\epsilon}) &\calO(n^{-1+\epsilon}) \\
\textup{Scr.\ Sobol' Sampling} &\calO(n^{-3/2+\epsilon}) & \FJHCNF^\textup{R} 
&& & &\calO(n^{1/2+\epsilon})&\calO(1) 
	\end{array}
	\end{equation*}
\end{table}

For IID sampling both the root 
mean square $L^2$-discrepancy and the randomized discrepancy are 
$\calO(n^{-1/2})$.  The confounding for typical IID sampling is $\calO(1)$.  In the 
randomized setting one may have an atypically poor instance of data sites 
that leads to an atypically high confounding of $\calO(n^{1/2})$.  On the other hand,  
unscrambled Sobol' sampling and scrambled Sobol' sampling are atypically superior 
instances 
of data sites under an IID sampling measure that yield atypically small confoundings of  
$\calO(n^{-1/2+\epsilon})$ and $\calO(n^{-1+\epsilon})$, respectively.

For scrambled Sobol' sampling, the root 
mean square $L^2$-discrepancy is now only $\calO(n^{-1+\epsilon})$, an improvement 
over IID sampling.  However, the 
randomized discrepancy is an even smaller 
$\calO(n^{-3/2+\epsilon})$ \cite{HeiHicYue02a,Owe97}. 
In the deterministic setting, 
unscrambled Sobol' sampling has a typical $\calO(1)$ confounding, whereas typical 
scrambled Sobol' sampling has an atypically low $\calO(n^{-1/2})$ confounding.   This is 
because 
scrambled Sobol' sampling can take advantage of the additional smoothness of the given
integrand, which is not reflected in the definition of $\calF$.  In the randomized setting,  
unscrambled Sobol' sampling 
has an atypically high $\calO(n^{1/2})$ confounding.
Thus, unscrambled Sobol' sampling is among the awful minority of sampling measures 
under scrambled Sobol' sampling.

\begin{FJHLesson}
\FJHLessonThree
\end{FJHLesson}

\begin{FJHLesson} \label{FJH:Lesson:Seven}
	\FJHLessonSeven
\end{FJHLesson}

An alternative to the Genz transformation above is an affine transformation to compute 
the multivariate Gaussian probability:
\begin{gather*}
\bsz = \bsa + (\bsb-\bsa) \circ \bsx, \quad f_{\textup{aff}}(\bsx) =  \frac{\exp\bigl(- 
\frac 12 \bsz^T
\mathsf{\Sigma}^{-1} \bsz 
	\bigr)}{\sqrt{(2 \pi)^d \det(\mathsf{\Sigma})}} \, \prod_{j=1}^d (b_j - a_j), \\
\nonumber
\mu = 
	\int_{[0,1]^{d}} f_{\textup{aff}}(\bsx) \, \D \bsx,
\end{gather*} 
where $\circ$ denotes the Hadamard (term-by-term) product.  The right plot in Fig.\ 
\ref{FJH:fig:MVNfig} shows that the error using the affine 
transformation is much worse than that using the Genz transformation even though the 
two convergence rates are the same.  The difference in the magnitudes of the errors is 
primarily because $\FJHVAR^{\textup{D}}(f_{\textup{aff}})$ is greater than 
$\FJHVAR^{\textup{D}}(f_{\textup{Genz}})$.


\begin{FJHLesson} \label{FJH:eq:lessonfour}
	\FJHLessonFour
\end{FJHLesson}

\section{Option Pricing} \label{FJH:sec:OptPrice}
The prices of financial derivatives can often be modeled by high dimensional integrals.  If 
the underlying asset is described in terms of a Brownian motion, $B$, at times $t_1, 
\ldots, t_d$, then $\bsZ = (B(t_1), \ldots, B(t_d)) \sim \calN(\bszero, \mathsf{\Sigma})$, 
where $\mathsf{\Sigma}  = \bigl( \min(t_j,t_k) \bigr)_{j,k=1}^d$, and the fair price of the 
option is
\begin{equation*}
\mu = \int_{\R^d} \textup{payoff}(\bsz) \, \frac{\exp\bigl(- \frac 12 \bsz^T 
\mathsf{\Sigma}^{-1} 
\bsz 
\bigr)}{\sqrt{(2 \pi)^d \det(\mathsf{\Sigma})}} \, \D \bsz = \int_{[0,1]^{d}} f(\bsx) \, \D \bsx ,
\end{equation*}
where the function $\textup{payoff}(\cdot)$ describes the discounted payoff of the 
option, 
\begin{equation*}
f(\bsx) = \textup{payoff}(\bsz), \qquad \bsz = \mathsf{L} \begin{pmatrix}
\Phi^{-1}(x_1) \\ \vdots \\ \Phi^{-1}(x_d)
\end{pmatrix}.
\end{equation*}
In this example, $\mathsf{L}$ may be any square matrix satisfying $\mathsf{\Sigma} = 
\mathsf{L} 
\mathsf{L}^T$.

Fig.\ \ref{FJH:fig:AsianOpt} shows the cubature error using IID sampling, unscrambled 
Sobol' sampling, and scrambled Sobol' sampling for the Asian arithmetic mean call option 
with the following parameters:
\begin{gather*}
\textup{payoff}(\bsz) = \max \left( \frac 1d \sum_{j=1}^d S_j - K, 0\right)\E^{-r\tau}, \quad
S_j = S_0 \exp\bigl( (r - \sigma^2/2) t_j + \sigma z_j\bigr),
\\
\tau = 1, \quad d = 12, \quad S_0 = K = 100, \quad r =  0.05, \quad \sigma = 0.5, \\
t_j = j \tau /d, \quad j = 1\!:\!d.
\end{gather*}
The convergence rates for IID and unscrambled Sobol' sampling are the same as in Fig.\ 
\ref{FJH:fig:MVNfig} for the previous example of multivariate probabilities.  However, 
for this example scrambling the Sobol' set improves the accuracy but not the 
convergence rate. The convergence rate for scrambled Sobol' sampling, 
$\widehat{\nu}$,  is 
poorer than hoped for 
because 
$f$ is not smooth enough for $\FJHVAR^{\textup{D}}(f)$ to be finite in the case 
where $\FJHDSC^\textup{R}(\nu - \widehat{\nu}) = \calO(n^{-3/2 +\epsilon})$.

\begin{figure}
	\centering
		\includegraphics[height = \FJHfigheight] 
		{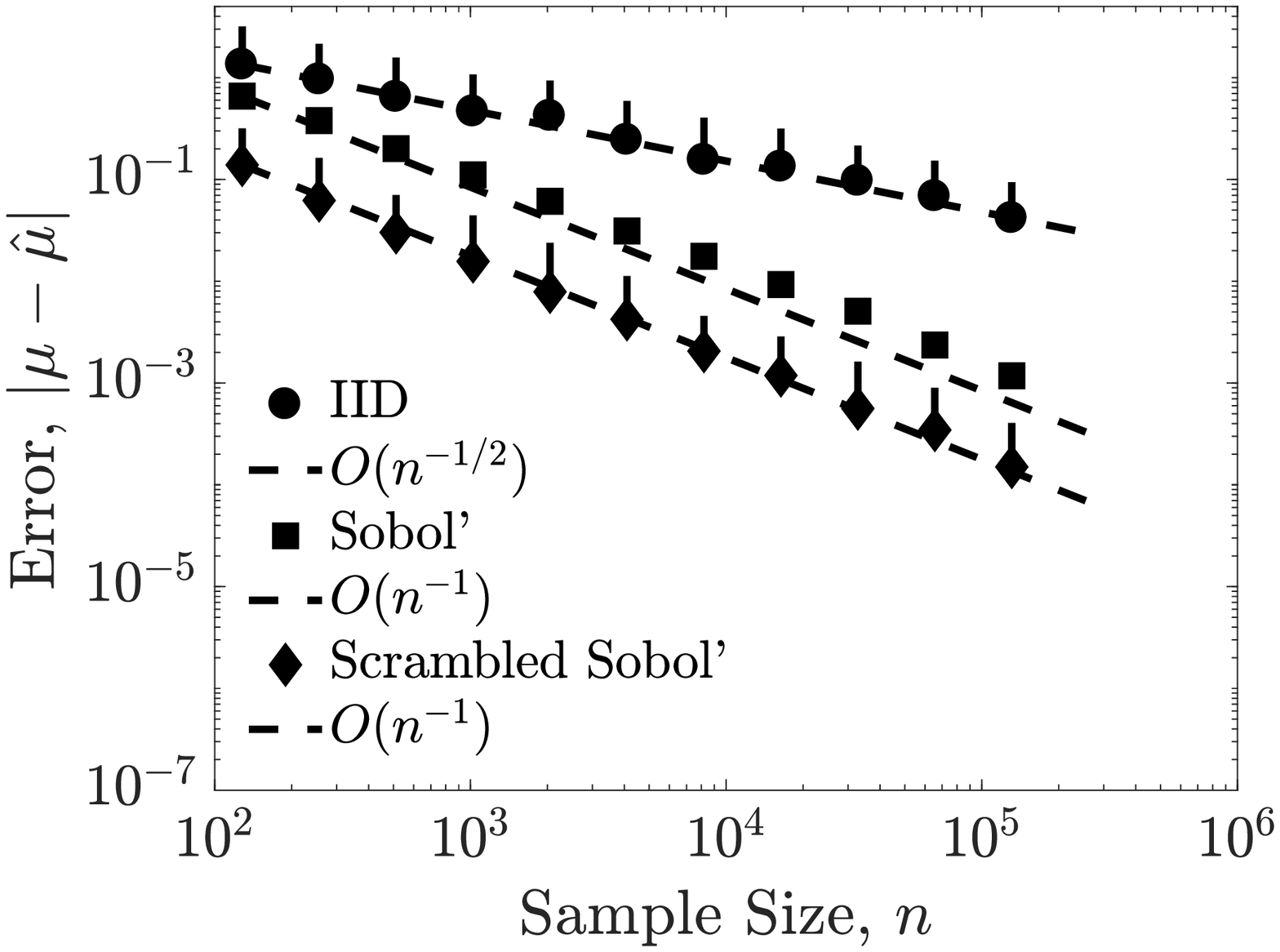} 
		\qquad
		\includegraphics[height = \FJHfigheight] 
		{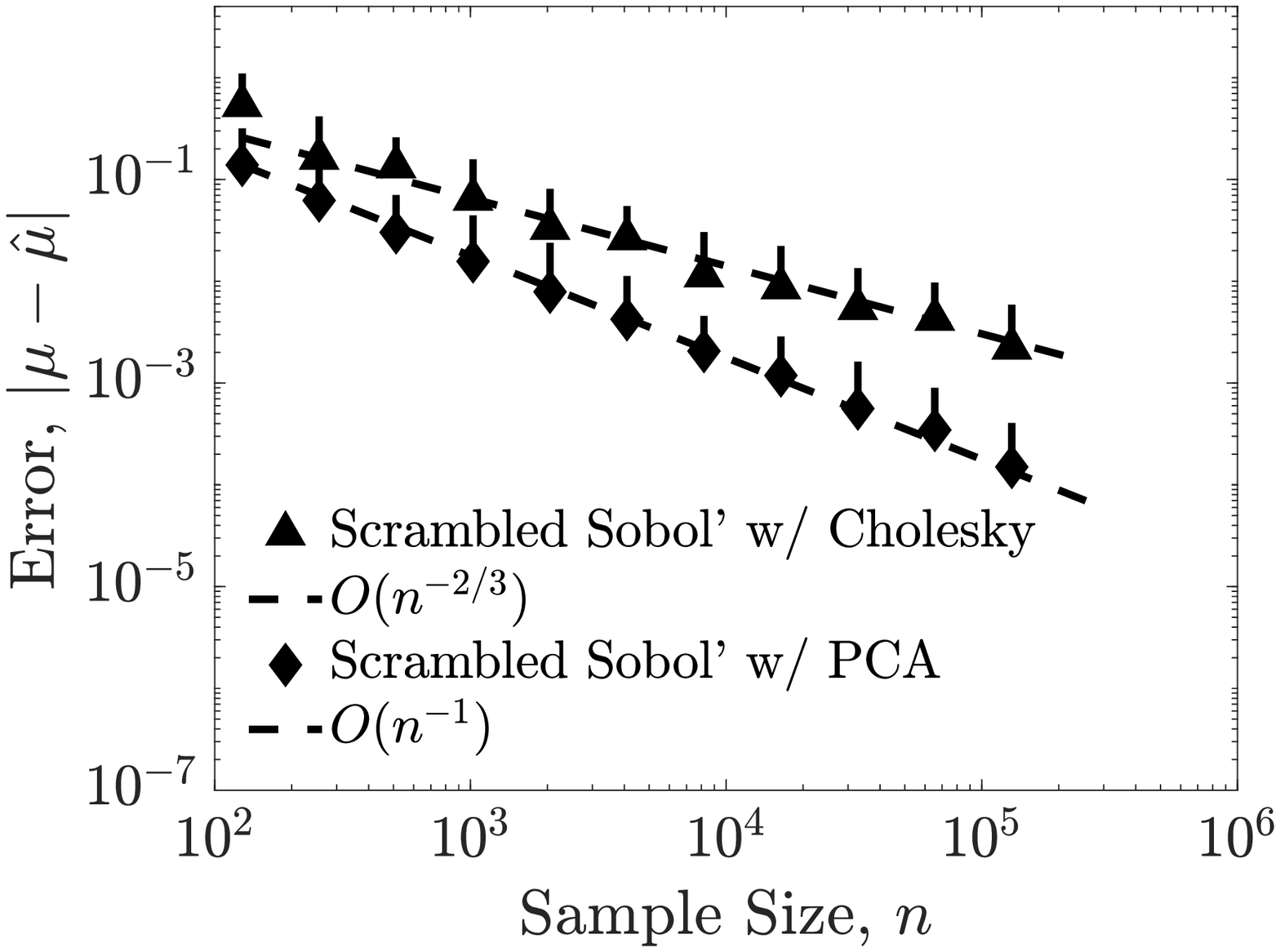}
		\caption{Cubature error for the price of an Asian arithmetic mean option using 
		different sampling 
		measures. The left side uses the PCA decomposition and the right side contrasts 
		the PCA with the Cholesky decomposition.  \label{FJH:fig:AsianOpt}}
\end{figure}

\begin{FJHLesson}
	\FJHLessonFive
\end{FJHLesson}

The left plot in Fig.\ \ref{FJH:fig:AsianOpt} chooses $\mathsf{L} = 
\mathsf{V}\mathsf{\Lambda}^{1/2}$, where the columns of $\mathsf{V}$ are the 
normalized 
eigenvectors of $\mathsf{\Sigma}$, and the diagonal elements of the diagonal matrix  
$\mathsf{\Lambda}$ are the eigenvalues of $\mathsf{\Sigma}$.  This is also called a 
principal component analysis (PCA) construction.  The advantage is that the main part of 
the 
Brownian motion affecting the option payoff is concentrated in the smaller dimensions.  
The right plot of Fig.\ 
\ref{FJH:fig:AsianOpt} contrasts the cubature error for two choices of  $\mathsf{L}$:  
one 
chosen by the PCA construction and the other coming from the Cholesky decomposition 
of $\mathsf{\Sigma}$.  This latter choice corresponds to constructing the Brownian 
motion by time 
differences.  The Cholesky decomposition of $\mathsf{\Sigma}$ gives a poorer rate of 
convergence, illustrating again Lesson \ref{FJH:eq:lessonfour}.  The superiority of the 
PCA construction was 
observed in \cite{AcwBroGla97}.

\section{A Bayesian Trio  Identity for Cubature Error} \label{FJH:sec:BayesTrio}
An alternative to the deterministic integrand considered thus far is to assume that the 
integrand that is a stochastic process.  Random input functions have been hypothesized 
by 
Diaconis \cite{Dia88a}, O'Hagan \cite{OHa91a}, Ritter \cite{Rit00a}, Rasmussen and 
Ghahramani \cite{RasGha03a}, and others. Specifically, suppose that $f \sim 
\calG\!\!\calP (0, 
s^2C_{\bstheta})$, a zero mean Gaussian process.  The covariance of 
this  Gaussian process is $s^2C_{\bstheta}$, where $s$ is a scale parameter, and 
$C_{\bstheta}:\calX \times \calX \to \R$ is defined by a shape parameter $\bstheta$.  
The 
sample space for this Gaussian process, $\calF$, does not enter significantly into the 
analysis.  Define the vector space of measures 
\begin{equation*}
\calM = \left\{\eta :  \biggl \lvert \int_{\calX^2} C_\bstheta(\bsx,\bst) 
\, \eta(\D \bsx) \eta(\D \bst) \biggr \rvert < \infty, \ \ \biggl \lvert \int_{\calX} 
C_\bstheta(\bsx,\bst) 
\, \eta(\D \bst) \biggr \rvert < \infty \ \ \forall \bsx \in \calX \right \},
\end{equation*}
and let $C_{\bstheta}$ be such that $\calM$ contains both $\nu$ and the Dirac 
measures 
$\delta_{\bst}$ for all $\bst \in \calX$. 

For a Gaussian process, all vectors of linear functionals of $f$ have a multivariate 
Gaussian distribution.  It then follows that for a 
deterministic sampling measure, $\widehat{\nu} = \sum_{i=1}^n w_i \delta_{\bsx_i}$, the 
cubature 
error, $\mu - \widehat{\mu}$, is distributed as $\calN \bigl(0 , s^2(c_0 - 
2 \bsc^T \bsw + \bsw ^T \mathsf{C} \bsw) \bigr)$, where 
\begin{subequations} \label{FJH:eq:Cvecmatdef}
\begin{gather}
c_0  = \int_{\calX^2} C_\bstheta(\bsx,\bst) \, \nu(\D \bsx) \nu(\D \bst), \qquad \bsc = 
\biggl( 
\int_{\calX} 
C_\bstheta(\bsx_i,\bst) \,\nu(\D \bst) \biggr)_{i=1}^n, \\
\mathsf{C}  = \bigl( C_\bstheta(\bsx_i,\bsx_j) \bigr)_{i,j=1}^n, \qquad \bsw = \bigl(w_i 
\bigr)_{i=1}^n.
\end{gather}
\end{subequations}
The dependence of $c_0$, $\bsc$, and $\mathsf{C}$ on the shape parameter 
$\bstheta$ is 
suppressed in the 
notation for simplicity. We 
define the Bayesian variation, discrepancy and confounding as 
\begin{subequations} \label{FJH:eq:btriodef}
\begin{gather}
\label{FJH:eq:btriodefa}
\FJHVAR^{\textup{B}}(f)  = s, \qquad \FJHDSC^{\textup{B}}(\nu - \widehat{\nu}) = 
\sqrt{c_0 - 
	2 \bsc^T \bsw + \bsw ^T \mathsf{C} \bsw},\\
\FJHCNF^{\textup{B}}(f,\nu - \widehat{\nu}) :=\frac{ \displaystyle \int_{\calX} 
	f(\bsx) \, (\nu - \widehat{\nu})(\D\bsx)}{s \sqrt{c_0 - 2 \bsc^T \bsw + \bsw ^T 
	\mathsf{C} \bsw}}.
\end{gather}
\end{subequations}

\begin{theorem}[Bayesian Trio Error Identity]  \label{FJH:thm:btrio} Let the integrand be 
an instance of a zero mean Gaussian process with covariance $s^2C_{\bstheta}$ and 
that is drawn from a sample space 
$\calF$.  For the  variation, discrepancy, and 
	confounding defined in \eqref{FJH:eq:btriodef}, the following error identity holds: 
	\begin{equation} \tag{BTRIO} \label{FJH:eq:btrio}
	\mu - \widehat{\mu}  = \FJHCNF^\textup{B}(f,\nu - \widehat{\nu}) \, 
	\FJHDSC^\textup{B}(\nu - \widehat{\nu}) \, 
	\FJHVAR^{\textup{B}}(f) \quad 
	\text{almost surely}.
	\end{equation}
	Moreover, $\FJHCNF^\textup{B}(f,\nu - \widehat{\nu}) \sim \calN(0,1)$. 
\end{theorem}
\begin{proof}  Although $\int_{\calX} f(\bsx) \, \nu(\D \bsx)$ and $f(\bst) = \int_{\calX} 
f(\bsx) \, \delta_\bst (\D \bsx)$ may not exist for all $f \in \calF$, these two quantities 
exist 
almost surely because $\bbE_f [\int_{\calX} f(\bsx) \, \nu(\D \bsx)]^2 = s^2c_0$, and 
$\bbE_f [f(\bsx)]^2 = s^2C_{\bstheta}(\bsx,\bsx)$ are both well-defined and finite.  The 
proof of the 
Bayesian trio identity follows directly from the definitions above.  The 
distribution of the confounding follows from the distribution of the cubature error.
\end{proof}

The choice of cubature weights that 
minimizes the Bayesian discrepancy in \eqref{FJH:eq:btriodefa} is $\bsw = 
\mathsf{C}^{-1} \bsc$, which results in $\FJHDSC^{\textup{B}}(\nu - \widehat{\nu}) = 
\sqrt{c_0 - 
\bsc ^T 
\mathsf{C}^{-1} \bsc}$ and $\mu - \widehat{\mu} \sim \calN \bigl(0 , s^2(c_0 - \bsc ^T 
\mathsf{C}^{-1} \bsc) \bigr)$.  However, computing the weights requires $\calO(n^3)$ 
operations 
unless $\mathsf{C}$ has some special structure.  This computational cost is significant 
and may be a deterrent to the use of optimal weights unless the weights are 
precomputed. For smoother covariance functions, $C_\bstheta$, there is often a 
challenge of $\mathsf{C}$ being ill-conditioned.

The \emph{conditional} distribution of the cubature error, $\mu - \widehat{\mu}$, given 
the 
observed data $ \{f(\bsx_i )= y_i\}_{i=1}^n$ is $\calN \bigl( \bsy^T (\mathsf{C}^{-1}\bsc - 
\bsw) , s^2(c_0 - \bsc ^T \mathsf{C}^{-1} \bsc) \bigr)$.  To remove the bias one should 
again 
choose $\bsw = \mathsf{C}^{-1} \bsc$.  This also makes the conditional distribution of 
the cubature error the same as the unconditional distribution of the cubature error.

Because the cubature error is a normal random variable, we may use function 
values to perform useful inference, namely, 
\begin{equation} \label{FJH:eq:MLEBdOne}
\bbP_f\bigl[\FJHabs{\mu - \widehat{\mu}} \le 2.58 \,
\FJHDSC^{\textup{B}}(\nu - \widehat{\nu}) \FJHVAR^{\textup{B}}(f)\bigr] = 99\%.
\end{equation}
However, unlike our use of random sampling measures that are constructed via carefully 
crafted random number generators, there is no assurance that our integrand is actually 
drawn from a Gaussian process whose covariance we have assumed.  

The covariance function, $s^2 C_{\bstheta}$, should be estimated, and one way to do so 
is 
through 
maximum likelihood estimation (MLE), using the function values drawn for the purpose of 
estimating the integral.  The log-likelihood function for the data $ \{f(\bsx_i )= 
y_i\}_{i=1}^n$ is
\begin{align*}
\ell(s,\bstheta | \bsy) & = \log \left(\frac{\exp\left(- \frac 12 s^{-2} \bsy^T 
\mathsf{C}_{\bstheta}^{-1} \bsy\right)}{\sqrt{(2\pi)^n  \det(s^2 
\mathsf{C_{\bstheta}})}}\right)\\
& = - \frac 12 s^{-2} \bsy^T \mathsf{C}_{\bstheta}^{-1} \bsy -\frac 12  
\log \bigl(
\det(\mathsf{C_{\bstheta}}) \bigr) - \frac n2 \log(s^2) + \text{constants}.
\end{align*}
Maximizing with respect to $s^2$, yields the MLE scale parameter:
\[
s_{\textup{MLE}} =  \sqrt{\frac 1n \bsy^T \mathsf{C}_{\bstheta_{\textup{MLE}}}^{-1} \bsy}.
\]
Plugging this into the log likelihood leads to the MLE shape parameter: 
\[
\bstheta_{\textup{MLE}} =  \argmin_\bstheta \left[\frac 1n \log 
\bigl(\det(\mathsf{C_{\bstheta}}) 
\bigr) 
+ \log\bigl(\bsy^T \mathsf{C}_{\bstheta}^{-1} \bsy \bigr)  \right ],
\]
which requires numerical optimization to evaluate. Using MLE estimates, the probabilistic 
error 
bound in \eqref{FJH:eq:MLEBdOne} becomes
\begin{multline} \label{FJH:eq:MLEBdTwo}
\bbP_f\left[\FJHabs{\mu - \widehat{\mu}} \le 2.58 
\sqrt{\frac 1n \left(c_{0,\bstheta_{\textup{MLE}}} - \bsc_{\bstheta_{\textup{MLE}}} ^T 
	\mathsf{C}_{\bstheta_{\textup{MLE}}}^{-1} \bsc_{\bstheta_{\textup{MLE}}} \right)  
	\left(\bsy^T 
	\mathsf{C}^{-1}_{\bstheta_{\textup{MLE}}}\bsy\right)}\right] \\
 = 99\%.
\end{multline}
Note that the value of $\bstheta_{\textup{MLE}}$ and the above Bayesian cubature error 
bound  
is unchanged by replacing  $C_\bstheta$ by a positive multiple of itself.  

Let's revisit the multivariate normal probability example of Sec.\ \ref{FJH:sec:Gauss}, 
and perform Bayesian cubature with a covariance kernel with modest smoothness from 
the Mat\'ern  family:
\begin{equation} \label{FJH:eq:Matern}
C_\theta(\bsx,\bst) = \prod_{j=1}^d \left(1 + 
\theta\FJHabs{x_j-t_j}\right) \exp\left(-\theta\FJHabs{x_j-t_j}\right)
\end{equation}
Using $100$ randomly scrambled Sobol' samples, the 
Bayesian 
cubature method outlined above was used to compute the multivariate normal probability 
$\mu$.  We used MLE scale and shape parameters and  optimal 
cubature weights $\bsw = 
\mathsf{C}_{\theta_{{\textup{MLE}}}}^{-1} \bsc_{\theta_{{\textup{MLE}}}}$.  
The actual errors are plotted in Fig.\ \ref{FJH:fig:MVNcubMLE}, which also provides a 
contrast of the actual error and the probabilistic error bound.  This bound was correct 
about $83\%$ of the time.  Based on the smoothness of the integrand and the kernel, 
one might expect $\calO(n^{-2})$ convergence of the answer, but this is not clear from 
the numerical computations.

\begin{figure}
	\centering
\includegraphics[height = \FJHfigheight] 
{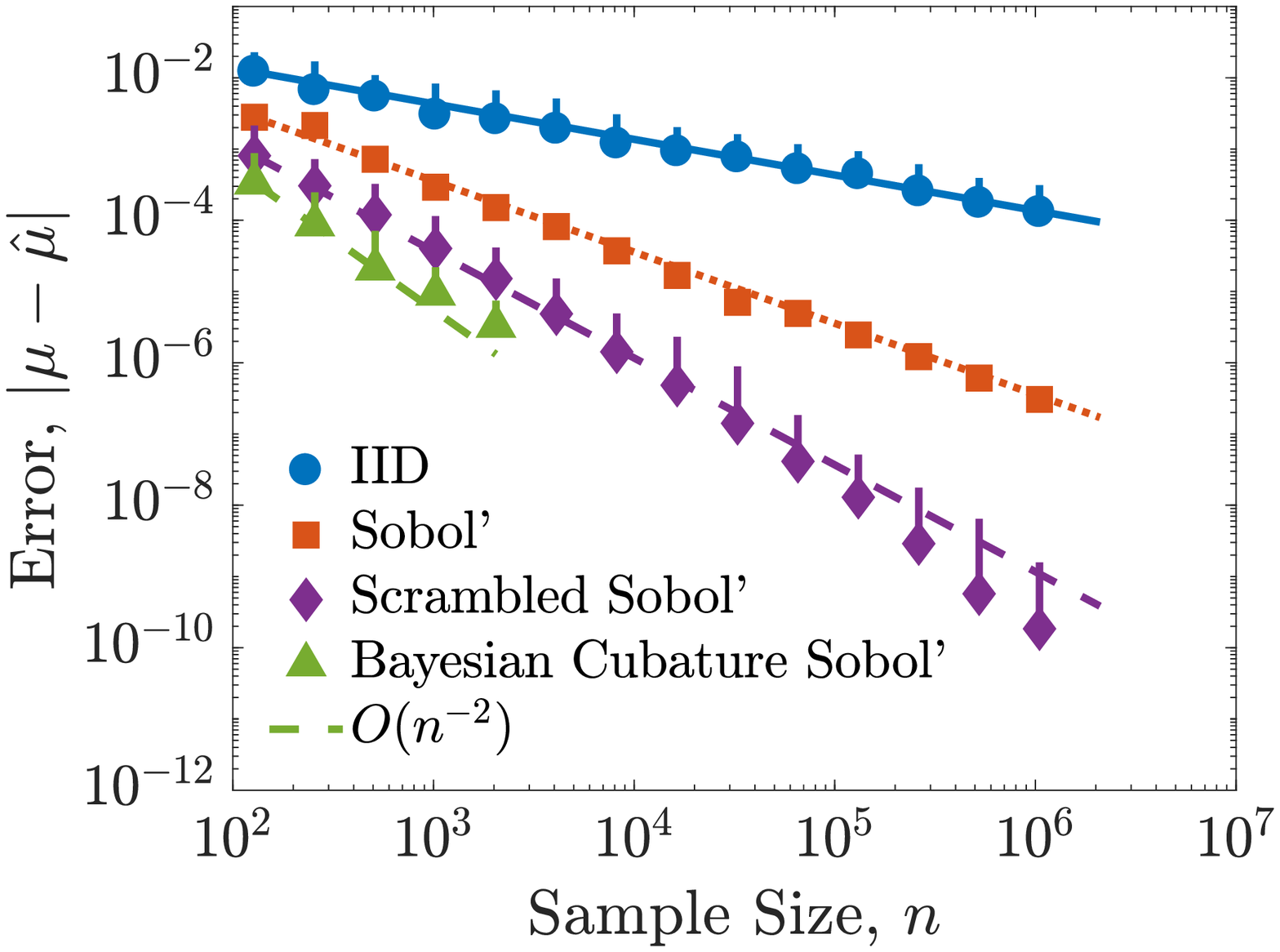}
\qquad
\includegraphics[height = \FJHfigheight] 
{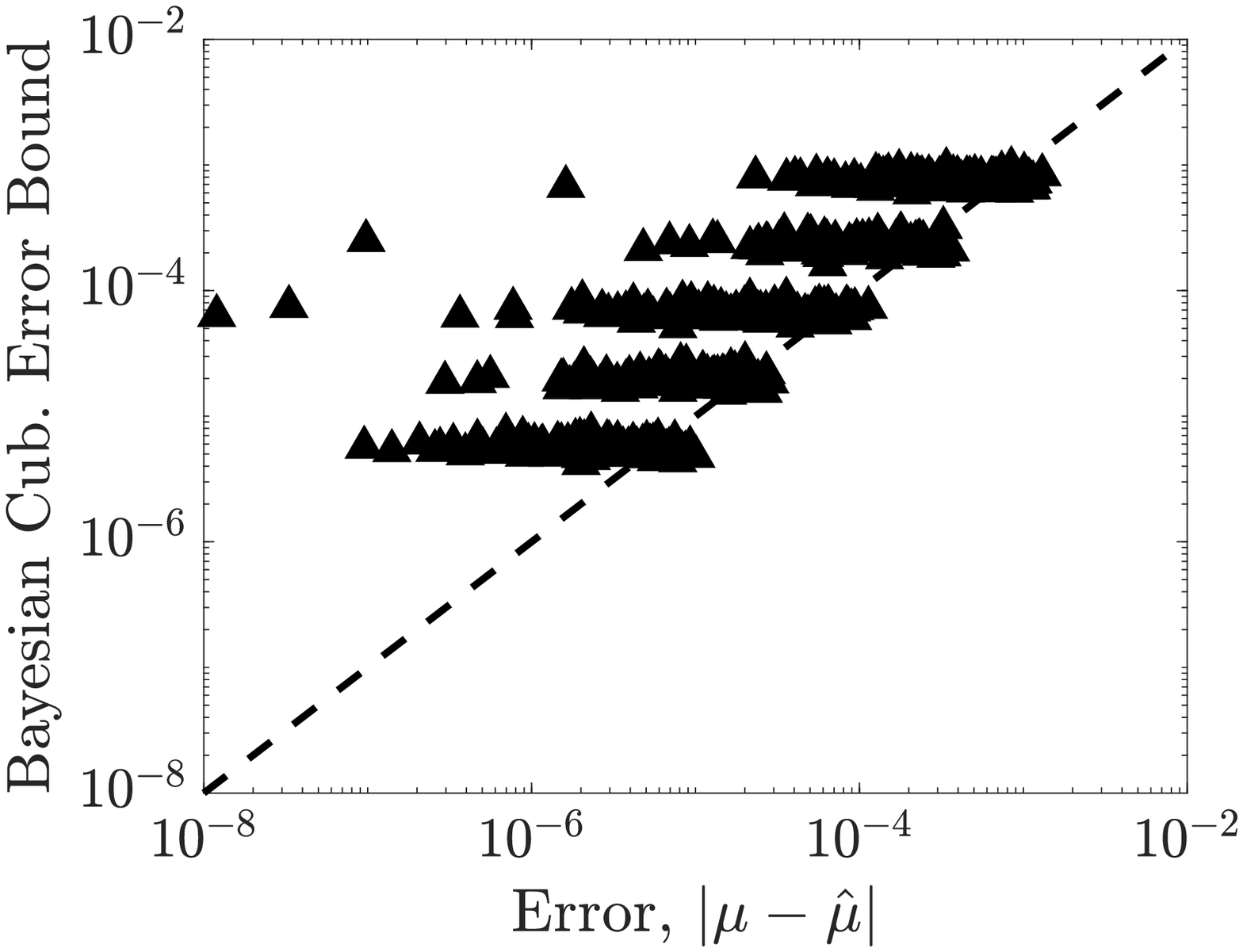}
\caption{The cubature errors for the 
multivariate normal 
probability example using Bayesian cubature (left), and the Bayesian cubature 
error versus the probabilistic error bound in \eqref{FJH:eq:MLEBdTwo} (right). 
\label{FJH:fig:MVNcubMLE}}
\end{figure}

\begin{FJHLesson}
	\FJHLessonTen
\end{FJHLesson}

Bayesian cubature offers hope with a dose of caution.  The theory is solid, but as  this 
example shows, one cannot know if the actual integrand under consideration is a typical 
instance of the Gaussian process being assumed, even when using MLE to determine the 
parameters of the distribution.  The success rate of the probabilistic error bound for this 
example is high, but not as high as the theory would suggest.  One may ask whether a 
larger candidate family of Gaussian processes needs to be considered, but then this 
might increase 
the time required for estimation of the parameters.  This example was carried 
out to only a rather modest sample size because of the $\calO(n^3)$ operations 
required to compute each $\widehat{\mu}$.  Efforts to reduce this operation count have 
been 
made by Anitescu, Chen, and Stein \cite{AniCheSte16a}, Parker, Reich and Gotwalt 
\cite{ParEtal17a}, and others.  Probabilistic numerics, 
\url{http://www.probabilistic-numerics.org}, of which Bayesian cubature is an example, 
holds promise that deserves further exploration.

The formulas for the Bayesian trio identity are analogous to those for the 
deterministic trio identity for reproducing kernel Hilbert spaces when $T(f) = 0$ for all $f 
\in \calF$.  Suppose that the reproducing kernel 
$K_\bstheta$ in the deterministic case is numerically equivalent to the covariance 
function $C_\bstheta$ used in Bayesian cubature.  The optimal cubature weights in the 
Bayesian case then mirror those in the 
deterministic case.  Likewise, for these optimal weights   
$\FJHDSC^\textup{D}(\nu - \widehat{\nu})$ is numerically the same as 
$\FJHDSC^\textup{B}(\nu 
- 
\widehat{\nu})$.

\begin{FJHLesson}
	\FJHLessonFourteen
\end{FJHLesson}

\section{A Randomized Bayesian Trio Identity for Cubature Error}

So far, we have presented three versions of the trio identity: a deterministic version  in 
Theorem \ref{FJH:thm:dtrio}, a randomized version in Theorem \ref{FJH:thm:rtrio}, and a 
Bayesian 
version in Theorem \ref{FJH:thm:btrio}.  The fourth and final version is a randomized 
Bayesian trio identity.  The variation remains unchanged from the Bayesian definition in 
\eqref{FJH:eq:btriodefa}.  
The randomized Bayesian discrepancy and confounding are defined as follows:
\begin{subequations} \label{FJH:eq:rbtriodef}
	\begin{gather}
\FJHDSC^{\textup{R}\textup{B}}(\nu - \widehat{\nu}) = \sqrt{\bbE_{\widehat{\nu}}  
\bigl(c_0 - 
		2 \bsc^T \bsw + \bsw ^T \mathsf{C} \bsw \bigr)},\\
	\FJHCNF^{\textup{R}\textup{B}}(f,\nu - \widehat{\nu}) :=\frac{ \displaystyle \int_{\calX} 
		f(\bsx) \, (\nu - \widehat{\nu})(\D\bsx)}{s  \sqrt{\bbE_{\widehat{\nu}}  \bigl(c_0 - 
			2 \bsc^T \bsw + \bsw ^T \mathsf{C} \bsw \bigr)}}.
	\end{gather}
\end{subequations}
The proof of the randomized Bayesian trio error identity is similar to the proofs of the 
other trio identities and is omitted.

\begin{theorem}[Randomized Bayesian Trio Error Identity]  \label{FJH:thm:rbtrio} Let 
the integrand be 
	an instance of a zero mean Gaussian process with covariance $s^2C_{\bstheta}$ and 
	that is drawn from a sample space $\calF$.  Let the sampling measure be drawn 
	randomly from $\calM_{\textup{S}}$ according to some probability distribution.  For 
	the  
	variation defined in \eqref{FJH:eq:btriodefa}, and the discrepancy and 
	confounding defined in \eqref{FJH:eq:rbtriodef}, the following error identity holds: 
	\begin{equation} \tag{RBTRIO} \label{FJH:eq:rbtrio}
	\mu - \widehat{\mu}  = \FJHCNF^{\textup{R}\textup{B}}(f,\nu - \widehat{\nu}) \, 
	\FJHDSC^{\textup{R}\textup{B}}(\nu - \widehat{\nu}) \, 
	\FJHVAR^{\textup{B}}(f) 
	\quad 
	\text{almost surely}.
	\end{equation}
	Moreover, $\FJHCNF^{\textup{R}\textup{B}}(f,\nu - \widehat{\nu}) \sim \calN(0,1)$. 
\end{theorem}

\begin{FJHLesson}
	\FJHLessonSix
\end{FJHLesson}

\section{Dimension Dependence of the Discrepancy, Cubature Error and 
Computational Cost} 
\label{FJH:sec:Tract}
The statements about the rates of decay of discrepancy and cubature error as the 
sample size increases have so far hidden the dependence on the dimension of the 
integration domain.  Fig.\ \ref{FJH:fig:AsianOpt} on the left shows a 
clear error 
decay rate of $\calO(n^{-1 + \epsilon})$ for low discrepancy sampling for the option 
pricing problem with dimension $12$.  However, Fig.\ \ref{FJH:fig:unwtdiscdiffpts} shows 
that the discrepancy for these scrambled Sobol' points does not decay as quickly as 
$\calO(n^{-1 
+ \epsilon})$ for moderate $n$.  

There has been a tremendous effort to 
understand the effect of the dimension of the integration problem on the convergence 
rate.  Sloan and Wo\'zniakowski \cite{SloWoz97} pointed out how
the sample size required to achieve a desired error tolerance could grow exponentially 
with 
dimension.  Such problems are called \emph{intractable}.  This led to a search for 
settings where the sample size required to  achieve a desired error tolerance only 
grows polynomially with dimension (\emph{tractable} problems) or is independent of the 
dimension (\emph{strongly tractable} problems). The three volume masterpiece by Novak 
and Wo\'zniakowski \cite{NovWoz08a,NovWoz10a,NovWoz12a} and the references cited 
therein contain necessary and sufficient conditions for tractability.  The parallel idea
of \emph{effective dimension} was introduced by Caflisch, Morokoff, 
and Owen \cite{CafMorOwe97} and developed further in \cite{LiuOwe04a}.

Here we provide a glimpse into those situations where the dimension of the problem does 
not have an adverse effect on the convergence rate of the cubature error and the 
discrepancy.  Let's generalize the reproducing kernel used to define the 
$L^2$-discrepancy in \eqref{FJH:eq:L2Kdef}, as well as the corresponding 
variation and the discrepancy for equi-weighted sampling measures by introducing 
\emph{coordinate weights} $\gamma_1, \gamma_2, \ldots$:
\begin{gather*} \label{FJH:eq:wtL2Kdef}
K(\bsx,\bst) =\prod_{k = 1}^d [1 + \gamma_k^2\{1 - \max(x_k,t_k)\}], \\
\FJHVAR^\textup{D}(f) = \Bigl \lVert \bigl (\gamma_{\fraku}^{-1}
\lVert\partial^{\fraku} f\rVert_{L^2}\bigr )_{\fraku \ne \emptyset} \Bigr \rVert_2 \qquad
\gamma_{\fraku} = \prod_{k \in \fraku} \gamma_k,
\end{gather*}
\begin{multline} \label{FJH:eq:wtL2disc}
\bigl[\FJHDSC^\textup{D}(\nu - \widehat{\nu})\bigr]^2 = \prod_{k=1}^d\Bigl(1 + 
\frac{\gamma_k^2}{3}\Bigr)
- \frac{2}{n} \sum_{i=1}^n \prod_{k=1}^d \left (1 + \frac{\gamma_k^2(1 - 
x_{ik}^2)}{2} \right) \\ + \frac{1}{n^2}\sum_{i,j=1}^n \prod_{k = 1}^d [1 + 
\gamma_k^2( 1 - \max(x_{ik},x_{jk}))].
\end{multline}
For $\gamma_1 = \cdots = \gamma_d = 1$, we recover the situation in Sec.\ 
\ref{FJH:sec:dettrio}, where the decay rate of the discrepancy is dimension dependent 
for moderate sample sizes.  However if $\gamma_k^2 = 
k^{-3}$, then the discrepancies for randomly shifted lattice nodesets and scrambled   
Sobol' sequences show only a slight dimension dependence, as shown in Fig.\ 
\ref{FJH:fig:wtdiscdiffpts}.

\begin{figure}
	\centering
	\includegraphics[height=\FJHfigheight]{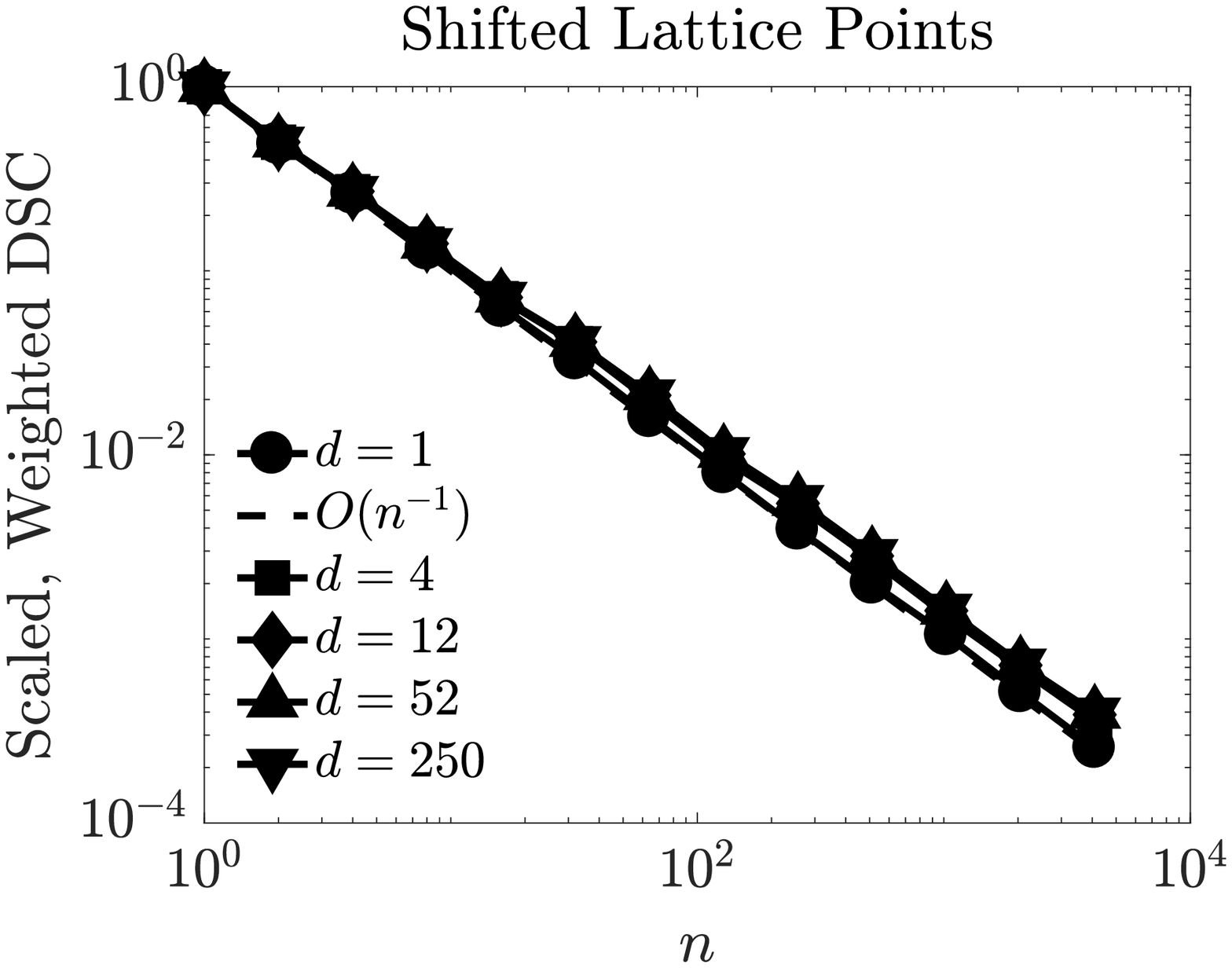}   \qquad 
	\includegraphics[height=\FJHfigheight]{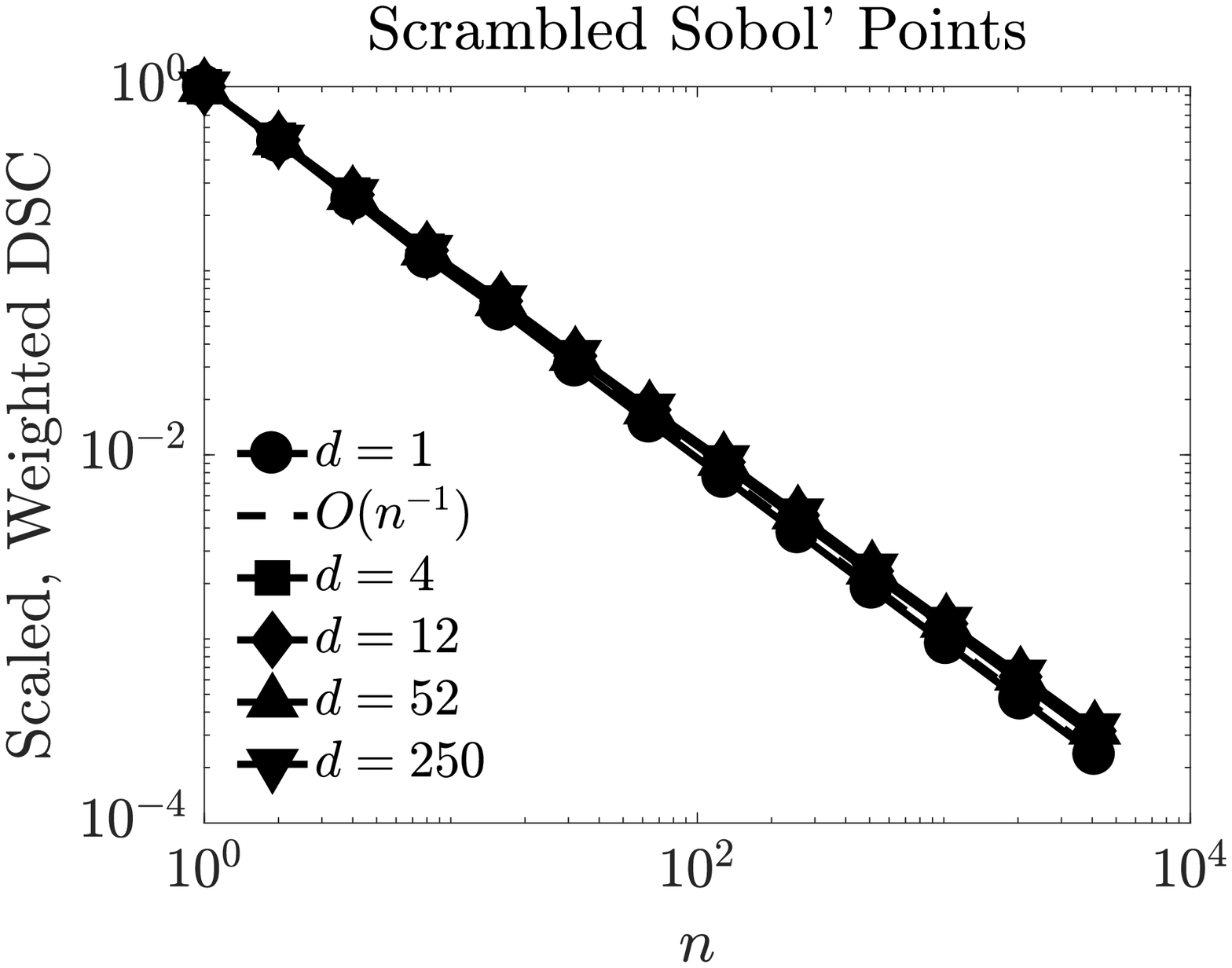} 
	\caption{The root mean square weighted $L^2$-discrepancies given by 
	\eqref{FJH:eq:wtL2disc} with $\gamma_k^2 = 
	k^{-3}$
		for randomly shifted 
		lattice sequence nodesets and randomly scrambled and shifted Sobol' sequences 
		points.  A variety of dimensions is shown.
		\label{FJH:fig:wtdiscdiffpts}}
\end{figure}

When the weights $\gamma_k$ decay with $k$, the discrepancy depends \emph{less} 
on how evenly the data sites appear in projections involving the higher numbered 
coordinates.  On 
the other hand, the variation in this case gives 
\emph{heavier} weight to the $\partial^{\fraku} f$ with 
$\fraku$ containing large $k$.  For the cubature error decay to mirror the decay of the 
discrepancy shown in Fig.\ \ref{FJH:fig:wtdiscdiffpts}, the integrand must depend only 
slightly on the coordinates with higher indices, so that the variation will be modest.

\begin{FJHLesson} 
	\FJHLessonEight
\end{FJHLesson}

For some integration problems the dimension is infinite and so our problem 
\eqref{FJH:eq:INT} 
becomes
\begin{equation} \label{FJH:eq:InfINT} \tag{$\infty$INT}
\mu = \lim_{d \to \infty} \mu^{(d)}, \qquad \mu^{(d)} = \int_{\calX^{(d)}} f^{(d)}(\bsx) \, 
\nu^{(d)}(\D 
\bsx),
\end{equation} 
where $\calX^{(d)} = \calX_1 \times \cdots \times \calX_d$, $\nu^{(d)}$ is a measure on 
$\calX^{(d)} $ 
with 
independent marginals $\nu_k$ on $\calX_k$, and $f^{(1)}, f^{(2)}, \ldots$ are 
approximations to 
an 
infinite-dimensional integrand.  The discrepancy and cubature error analysis for $d \to 
\infty$ is 
similar to the large $d$ situation, but now the \emph{compuational cost} of the 
approximate integrand is a concern \cite{CreuEtal08a, HicMGRitNiu09a,  WanHic00b, 
NiuHic09b}.

One could approximate $\mu$ by $\widehat{\mu}^{(d)}$, the 
approximation 
to 
$\mu^{(d)}$, for some large $d$.  However, the computational cost of evaluating  
$f^{(d)}(\bsx)$ for a single $\bsx$ typically requires $\calO(d)$ 
operations.  So this approach would  require a high computational cost of  $\calO(nd)$ 
operations to 
compute $\widehat{\mu}^{(d)}$.

The often better alternative is to decompose the $f^{(d)}$ into pieces $f_\fraku$, for 
$\fraku 
\subset 
1 \!\!:\!\! d$, such that $f^{(d)} = \sum_{\fraku \subseteq 1:d} f_{\fraku}$ and the 
$f_\fraku$ depend on $\fraku$ but \emph{not on $d$}.  Multi-level Monte Carlo 
approximates \eqref{FJH:eq:InfINT} by
\begin{equation*}
\widehat{\mu} := \widehat{\mu}\bigl(f^{(d_1)}\bigr) + \widehat{\mu}\bigl(f^{(d_2)} - 
f^{(d_1)}\bigr) + 
\cdots + \widehat{\mu}\bigl(f^{(d_L)} - f^{(d_{L-1})}\bigr), 
\end{equation*}
for some choice of $d_l$ with $d_1 < \cdots < d_L$.  This works well when 
$\FJHVAR\bigl(f^{(d_l)} 
- f^{(d_{l-1} )} \bigr)$ decreases as $l$ increases and when $\mu - \mu^{(d_L)}$ is small 
\cite{
	Gil14a, Gil15a, Gil08b, Hei01a,  HicMGRitNiu09a, NiuHic09b}.  The computational cost 
	of 
$\widehat{\mu}\bigl(f^{(d_l)} - 
f^{(d_{l-1})} \bigr)$ is $\calO(n_l d_l)$, and as $d_l$ increases, $n_l$ decreases, thus 
moderating the cost.   There is bias, since 
$\mu - 
\mu^{(d_L)}$ is not approximated at all, but this can be removed by a clever a 
randomized 
sampling method \cite{RheGly12a}.

The Multivariate Decomposition Method approximates \eqref{FJH:eq:InfINT} by
\begin{equation*}
\widehat{\mu} = \widehat{\mu}(f_{\fraku_1}) + \widehat{\mu}(f_{\fraku_2}) + 
\cdots +\widehat{\mu}(f_{\fraku_L}),
\end{equation*} 
where the $\fraku_l$ are the important sets of coordinate indices as judged by 
$\FJHVAR^{\textup{D}}(f_\fraku)$ to ensure that $\mu - \sum_{\fraku \notin \{\fraku_1,   
\ldots, 
	\fraku_L\}} \mu(f_\fraku)$ is small \cite{Was13b}.  The computational cost of each 
$\widehat{\mu}(f_{\fraku_l})$ 
is $\calO(n_l \lvert \fraku_l \rvert)$.  If the important sets have small cardinality, 
$\lvert \fraku_l \rvert$, the computational cost is moderate.  

\begin{FJHLesson}
	\FJHLessonThirteen
\end{FJHLesson}

\section{Automatic Stopping Criteria for Cubature}
The trio identity decomposes the cubature error into three factors.  By improving the 
sampling scheme, the discrepancy may be made smaller.  By re-writing the integral, the 
variation of the integrand might be made smaller.  For certain situations, we may find that 
the confounding is small.  While the trio identity helps us understand what contributes to 
the cubature error, it does not directly answer the question of how many samples are 
required to achieve the desired accuracy, i.e., how to ensure that 
\begin{equation} \label{FJH:eq:errcrit} \tag{ErrCrit}
\lvert \mu - \widehat{\mu} \rvert \le \varepsilon
\end{equation}
for some predetermined $\varepsilon$.

Bayesian cubature, as described in Sec.\ \ref{FJH:sec:BayesTrio}, provides data-based 
cubature
error bounds.  These can be used to determine how large $n$ must be to satisfy 
\eqref{FJH:eq:errcrit} with high probability.

For IID Monte Carlo the Central Limit Theorem may be used to construct an approximate 
confidence interval for $\mu$, however, this approach relies on believing that $n$ is 
large enough to have i) reached the asymptotic limit, and ii) obtained a reliable upper 
bound on the standard deviation in terms of a sample standard deviation.  There have 
been recent efforts to develop a more robust approach to fixed width confidence 
intervals 
\cite{BayEtal14a,HicEtal14a,Jia16a}.  An 
upper bound on the standard deviation may be computed by assuming an upper bound 
on the kurtosis or estimating the kurtosis from data.  The standard 
deviation of an integrand can be confidently bounded in terms of the sample standard 
deviation if it lies in the cone of functions with a known bound on their kurtosis.
A bound on 
the kurtosis also allows one to use a Berry-Esseen inequality, which is a finite sample version 
of  the Central Limit Theorem, to determine a sufficient sample size for computing the integral 
with the desired accuracy.

For low discrepancy sampling, independent random replications may be used to estimate 
the error, but this approach lacks a rigorous justification.  An alternative proposed by the 
author and his collaborators is to decompose the integrand into a Fourier series and 
estimate the decay rate of the Fourier coefficients that contribute to the error
\cite{HicJim16a,HicEtal17a,JimHic16a}.  This approach may also be used to satisfy 
relative error criteria or error criteria involving a function of several integrals  
\cite{HicEtal17a}.  Our automatic stopping criteria have been 
implemented in the Guaranteed Automatic Integration Library (GAIL) \cite{ChoEtal15a}.

\begin{FJHLesson}
	\FJHLessonEleven
\end{FJHLesson}

\section{Summary}
To conclude, we repeat the lessons highlighted above.  The order may be somewhat 
different.

\FJHLessonZero \FJHLessonSix  \FJHLessonTwoHalf \FJHLessonFourteen 
\FJHLessonSeven 

\bigskip

\hangindent = 5ex
\noindent
\emph{\FJHQOne} \FJHLessonTwo \FJHLessonThree   \FJHLessonFive

\hangindent = 5ex
\noindent
\emph{\FJHQTwo} \FJHLessonFour 	\FJHLessonEight 	\FJHLessonThirteen

\hangindent = 5ex
\noindent
\emph{\FJHQThree} \FJHLessonTen \FJHLessonEleven

\begin{acknowledgement}
	The author would like to thank the organizers of MCQMC 2016 for an exceptional 
	conference.  The author is indebted to his colleagues in the MCQMC 
	community for all that he has learned from them.  In particular, the author thanks 
 Xiao-Li Meng for introducing the trio identity and for 
 discussions related to its development.  The author also thanks Llu\'is Antoni 
	Jim\'enez Rugama for helpful comments in preparing this tutorial.  This work is partially 
	supported by the National Science Foundation grant DMS-1522687.
\end{acknowledgement}


\FJHrestorevalues

\end{document}

%% file: macros.tex
\renewcommand{\email}[1]{\emailname: #1} 

\usepackage{mathptmx}       
\usepackage{helvet}         
\usepackage{courier}        

\usepackage{makeidx}         
\usepackage{graphicx}        
\usepackage[bottom]{footmisc}

\usepackage{latexsym}
\usepackage{amsmath}
\usepackage{amsfonts}
\usepackage{amssymb}
\usepackage{bm}

\usepackage{url}
\usepackage{algorithm}
\usepackage{algorithmic}
\usepackage[misc,geometry]{ifsym}

\renewenvironment{proof}{\noindent{\itshape Proof.}}{\smartqed\qed}

\spdefaulttheorem{assumption}{Assumption}{\upshape \bfseries}{\itshape}
\spdefaulttheorem{algo}{Algorithm}{\upshape \bfseries}{\itshape}

\newcommand{\bsa}{{\boldsymbol{a}}}
\newcommand{\bsb}{{\boldsymbol{b}}}
\newcommand{\bsc}{{\boldsymbol{c}}}

\newcommand{\bsk}{{\boldsymbol{k}}}

\newcommand{\bst}{{\boldsymbol{t}}}

\newcommand{\bsw}{{\boldsymbol{w}}}
\newcommand{\bsx}{{\boldsymbol{x}}}
\newcommand{\bsy}{{\boldsymbol{y}}}
\newcommand{\bsz}{{\boldsymbol{z}}}

\newcommand{\bsX}{{\boldsymbol{X}}}

\newcommand{\bsZ}{{\boldsymbol{Z}}}

\newcommand{\bszero}{{\boldsymbol{0}}} 
\newcommand{\bsone}{{\boldsymbol{1}}}  

\newcommand{\bstheta}{{\boldsymbol{\theta}}}

\newcommand{\bbE}{{\mathbb{E}}}

\newcommand{\bbP}{{\mathbb{P}}}

\newcommand{\N}{{\mathbb{N}}} 
\newcommand{\R}{{\mathbb{R}}} 

\DeclareSymbolFont{bbold}{U}{bbold}{m}{n}
\DeclareSymbolFontAlphabet{\mathbbold}{bbold}

\newcommand{\calF}{{\mathcal{F}}}
\newcommand{\calG}{{\mathcal{G}}}

\newcommand{\calM}{{\mathcal{M}}}
\newcommand{\calN}{{\mathcal{N}}}
\newcommand{\calO}{{\mathcal{O}}}
\newcommand{\calP}{{\mathcal{P}}}

\newcommand{\calX}{{\mathcal{X}}}

\newcommand{\fraku}{{\mathfrak{u}}}

\providecommand{\argmin}{\operatorname*{argmin}}